\documentclass[pdflatex,sn-mathphys-num]{sn-jnl}

\usepackage{graphicx}%
\usepackage{multirow}%
\usepackage{amsmath,amssymb,amsfonts}%
\usepackage{amsthm}%
\usepackage{mathrsfs}%
\usepackage[title]{appendix}%
\usepackage{xcolor}%
\usepackage{textcomp}%
\usepackage{manyfoot}%
\usepackage{booktabs}%
\usepackage{algorithm}%
\usepackage{algorithmicx}%
\usepackage{algpseudocode}%
\usepackage{listings}%

\theoremstyle{thmstyleone}%
\newtheorem{theorem}{Theorem}
\theoremstyle{thmstyletwo}%
\newtheorem{remark}{Remark}%

\theoremstyle{thmstylethree}%

\theoremstyle{thmstylethree}%
\newtheorem{corollary}{Corollary}%

\begin{document}

\title[Sparse entropic quadrature for moment equations]{Sparse entropic quadrature for moment equations}


\author*[1]{\fnm{Georgii} \sur{Oblapenko}}\email{oblapenko@acom.rwth-aachen.de}

\author[2]{\fnm{Michael} \sur{Herty}}\email{herty@igpm.rwth-aachen.de}

\author[1]{\fnm{Manuel} \sur{Torrilhon}}\email{mt@acom.rwth-aachen.de}

\affil*[1]{\orgdiv{Applied and Computational Mathematics}, \orgname{RWTH Aachen}, \orgaddress{\street{Schinkelstrasse 2}, \city{Aachen}, \postcode{52062}, \country{Germany}}}

\affil[2]{\orgdiv{Institut für Geometrie und Praktische Mathematik}, \orgname{Organization}, \orgaddress{\street{Im Süsterfeld 2}, \city{Aachen}, \postcode{52072}, \country{Germany}}}


\abstract{In the present work, we study the properties of the sparse entropic quadrature method, a recently proposed method
for the closure of systems of moment equations in rarefied gas dynamics. We prove the properties of the method regarding its well-posedness, hyperbolicity, and stability for first-order spatial convection schemes. We apply the method to two model rarefied gas dynamics problems and compare the results with those given by a discrete velocity method and show that it is capable of reproducing rarefied gas effects whilst having significantly fewer degrees of freedom than the discrete velocity method.}




\maketitle
\section{Introduction}\label{sec:introduction}
Modelling of rarefied gas effects, that is, effects which are not captured by classical Euler and Navier--Stokes equations, is required
for simulation of a wide variety of problems in areas such as aerospace engineering, design of microelectronic devices, and astrophysics~\cite{cercignani1988,shen2005rarefied}.
Approaches to model such flows include particle-based methods such as DSMC and stochastic particle Fokker-Planck methods, spectral methods, discrete velocity methods, and moment methods. The latter family of methods possesses some attractive properties, such
as being amenable to inclusion in existing computational fluid dynamics codes, being capable of the Euler and Navier--Stokes equations in the continuum limit, and possibility for model adaptivity on the equation level~\cite{torrilhon2017hierarchical,verbiest2026model}. However, for the closure of the moment system of equations, one needs to estimate the higher-order moments of the distribution function without having knowledge about the distribution function itself, i.e. solving the classical moment closure problem~\cite{hamburger1944hermitian,Shohat1945problem,schmuedgen2017moment}.
Moment methods also arise in other contexts, such as radiative transport~\cite{struchtrup1998number,modest2021radiative}, geosciences~\cite{milbrandt2005multimoment,yuan2012extended,koellermeier2020analysis}, population dynamics~\cite{singh2006moment,gillespie2009moment}, and traffic flow modelling.~\cite{marques2013kinetic,herty2020bgk}

A large body of work exists on solving the moment closure problem, with various approaches exhibiting different advantages and disadvantages. Frequently, the distribution function is expanded into a finite ansatz, with the expansion coefficients, as well as potentially the parameters governing the basis functions, being found by matching the moments of the ansatz to the known set of transported moments. However, care must be taken to ensure that a non-negative distribution can be reconstructed, the resulting moment transport
equations are hyperbolic, the problem is not ill-conditioned, and that the approach is computationally efficient.
In addition, many methods have been developed to solve the moment closure problem in one dimension, but the extension to multiple dimensions is not always straighforward, and also imposes additional constraints, such as rotational invariance of the resulting system.

Grad's method~\cite{grad2kinetic} and its regularized version~\cite{struchtrup2003regularization}, quadrature-based approaches~\cite{mcgraw1997description,fox2008quadrature,desjardins2008quadrature,chalons2010beyond,fox2018conditional,van2021higher,huang2020stability,FoxLaurent,fox2023generalized,yilmaz2026nonlinear}, and maximum entropy methods~\cite{levermore1997entropy,mcdonald2013affordable,alldredge2019regularized} are among the most popular methods to study and solve the moment problem. Some methods, such as the projection method~\cite{koellermeier2014framework} or the $\varphi$-divergence formulation~\cite{abdelmalik2016moment} provide a more unified view of the problem, as well as generalization pathways; a recent study~\cite{pichardconvergence} investigates the question of convergence of the moment equations to the underlying kinetic equations. We also refer the reader to~\cite{torrilhon2016modeling,pichard2023some} for overviews of moment methods for kinetic theory problems.

Other more recent approaches include methods which can be classified as hybrid, as they utilize ansatzes from one approach with optimization objectives of another. Such methods include the maximum entropy formulation of Schaerer et al.~\cite{schaerer2017efficient,schaerer201735}, which uses quadrature rules to evaluate the integrals appearing in the entropy minimization problem and is directly comparable to discrete velocity methods; the entropic quadrature method~\cite{bohmer2020entropic}, which combines the maximum entropy method with a quadrature-based moment approach to construct the quadratures; and other related approaches such as the augmented discrete velocity method, which uses a fixed set of quadrature points~\cite{pichard2025entropy}, and the sparse entropic quadrature method~\cite{oblapenko2026sparse},
which augments the objective with a sparsity-promoting term and uses the local equilibrium distribution as a preconditioner. It should be noted that due to the finite support of these hybrid methods, they avoid the issues of non-realizable states of the maximum entropy method~\cite{junk1998domain,hauck2008convex,schaerer2017efficient}. Moreover, unlike  quadrature-based methods, where extension to multiple dimension is non-trivial~\cite{yuan2011conditional,rice2026robustly}, these are easily applied to multi-dimensional moment systems, although at the cost of having to solve systems of equations whose number of variables grows exponentially with the number of dimensions; a recent approach circumvents this by using a particle-like ansatz~\cite{oblapenko2026sparsekrm}.
The enforcement of sparsity not only allows using the reconstruction method for efficient storage of high-resolution distributions,
but also for efficient coupling of moment methods with discrete velocity solvers~\cite{dimarco2014numerical}.

In the present work, we analyze the theoretical properties of the sparse entropic quadrature method, extending the analysis of the augmented discrete velocity method first developed in~\cite{pichard2025entropy}, and apply the method to model one-dimensional flows.
The paper is structured as follows. In~Sec.\ref{sec:equations}, we introduce the governing equations and the moment system. In~Sec.~\ref{sec:closure} we present the sparse entropic quadrature closure and analyze its properties --- first, we prove existence, uniqueness, and smoothness of the closure; then, we analyze hyperbolicity of the closure in the one-dimensional case and prove strict hyperbolicity.
We then discuss numerical schemes for transport in the one-dimensional case, and prove realizability for first-order transport schemes, and present a Jacobian-free second-order Lax-Wendroff scheme.
In~Sec.~\ref{sec:algorithm} we provide an overview of the full numerical algorithm for solving collisional rarefied gas flows in one spatial dimension.
In~Sec.\ref{sec:numerical results} we present numerical results for the Sod shock tube and Couette flow test cases, and finally, in Sec.~\ref{sec:conclusions} we present our conclusions and avenues for future work.

\section{Governing equations}\label{sec:equations}

We are interested in numerically solving the following kinetic equation:
\begin{equation}
    \partial_t f(\mathbf{x},\mathbf{v},t) + \mathbf{v} \cdot \partial_{\mathbf{x}} f(\mathbf{x},\mathbf{v},t) = \mathcal{Q}(f(\mathbf{x},\mathbf{v},t)).\label{eq:kinetic}
\end{equation}
Here $f(\mathbf{x},\mathbf{v},t) \geq 0$ is the probability density function of particles at position $\mathbf{x} \in \mathbb{R}^{d_x}$ and velocity $\mathbf{v} \in \mathbb{R}^{d}$ at time $t \geq 0$, and $\mathcal{Q}$ is a collision operator describing the interaction of particles. 
In the present work, we are mostly concerned with the convective part appearing on the left-hand side of~\eqref{eq:kinetic}, and therefore specify $\mathcal{Q}$ only in the section on numerical results.

Due to the high dimensionality of $f$, direct solution of~\eqref{eq:kinetic} is computationally expensive. In this work, we consider moment methods, where we model the evolution of the moments of $f$, defined as
\begin{equation}
    m_{k_1,\ldots,k_d} = \int  f(\mathbf{v})\prod_{l=1}^{d} v_l^{k_l} \mathrm{d} \mathbf{v} \label{eq:constraints},
\end{equation}
where $k_l \geq 0$, $l=1,\ldots,d$.

From the kinetic equation~\eqref{eq:kinetic}, we can derive the following transport equation for the moments by multiplying it with $\prod_{l=1}^{d} v_l^{k_l}$ and integrating over $\mathbf{v}$:
\begin{equation}
    \frac{\partial m_{k_1,\ldots,k_d}}{\partial t} + \nabla \cdot \mathbf{F}_{k_1,\ldots,k_d} = \mathcal{Q}_{k_1,\ldots,k_d},\:k_l \geq 0,\: l=1,\ldots,d.\label{eq:transport}
\end{equation}
Here the flux is given by
\begin{align}
    \mathbf{F}_{k_1,\ldots,k_d} = \left( m_{k_1+1,k_2,\ldots,k_{d-1},k_d},m_{k_1,k_2+1,\ldots,k_{d-1},k_d},\ldots,\right.\nonumber\\
    \left. m_{k_1,k_2,\ldots,k_{d-1}+1,k_d},m_{k_1,k_2,\ldots,k_{d-1},k_d+1}\right),
\end{align}
and the source term as
\begin{equation}
    \mathcal{Q}_{k_1,\ldots,k_d} = \int \mathcal{Q}(f(\mathbf{x},\mathbf{v},t)) \prod_{l=1}^{d} v_l^{k_l} \mathrm{d} \mathbf{v}.
\end{equation}
As each equation for the time evolution of $m_{k_1,\ldots,k_d}$ involves higher-order moments of $f$, a closure is required, i.e. a way to compute $m_{k_1+1,k_2,\ldots,k_{d-1},k_d}$, $m_{k_1,k_2+1,\ldots,k_{d-1},k_d}$ given only knowledge about the lower-order moments.
We do that by reconstructing an underlying distribution function $f$ from the moments and using it to compute the missing higher-order moments required
to evaluate the fluxes.

\section{Sparse entropic quadrature closure}\label{sec:closure}
We use a discrete representation of the distribution function $f$ on a fixed grid in velocity space:
\begin{equation}
    f(\mathbf{x},\mathbf{v},t) \approx \sum_{i_1,\ldots,i_d}f_{i_1,\ldots,i_d}(\mathbf{x},t) \delta\left(\mathbf{v} - \mathbf{v}_{i_1,\ldots,i_d}\right).\label{eq:f-dvm}
\end{equation}
where $\mathbf{v}_{i_1,\ldots,i_d}$ are fixed velocity nodes, and $f_{i_1,\ldots,i_d}$ are the unknown values of the distribution function. The nodes have associated quadrature weights $\Delta v_{i_1,\ldots,i_d}$ that we use
when integrating functionals of $f$.
From now on, we consider the set of values of the distribution function as a long vector of values $\mathbf{f}\in \mathbb{R}^{N}$ (via a bijective mapping from $(i_1,\ldots,i_d) \to l$). Similarly, we ``unwrap'' the quadrature weights into a vector $\Delta \mathbf{v}$ with components $\Delta v_l$, and the velocity nodes into a vector $\mathbf{V}\in \left(\mathbb{R}^{d}\right)^N$ where each element $\mathbf{v}_l$ is a $d$-dimensional vector on the velocity grid.

We also define, for a given set of values $\mathbf{f}$, a discrete measure and the integration with respect to it:
\begin{equation}
    \mu(\mathbf{v}) := \sum_{l}\Delta v_{l}f_{l}\delta(\mathbf{v} - {\mathbf{v}_{l}}),\qquad
    \left\langle \cdot \right\rangle_{\mu} := \int \cdot\ \mathrm{d}\mu.
\end{equation}
Below, $\mu^{\ast}$ denotes the measure built in this way from the reconstruction $\mathbf{f}^{\ast}$ obtained by solving the optimization problem of
Section~\ref{sec:closure}, and we abbreviate $\left\langle \cdot \right\rangle := \left\langle \cdot \right\rangle_{\mu^{\ast}}$.

We can then introduce a moment measurement matrix $\mathbf{A}\in \mathbb{R}^{M \times N}$ which computes a vector of $M$ moment values
via integration against the discrete measure defined above, which reduces to a matrix-vector product:
\begin{equation}
    \mathbf{m} = \mathbf{A}\mathbf{f},\label{eq:moment-compute-linear-def}
\end{equation}
where the components of $\mathbf{A}$ are given by
\begin{equation}
    \mathbf{A}_{kl} = \Delta v_{l}p_{k}(\mathbf{v}_{l}),\qquad
    p_{k}(\mathbf{v}_l) := \prod_{i=1}^{d} v_{l,i}^{\alpha_{k,i}},\qquad k = 1,\ldots,M, \quad l = 1,\ldots,N,\label{eq:A-def}
\end{equation}
where $v_{l,i}$ is the $i$-th component of $\mathbf{v}_l\in\mathbb{R}^{d}$ and $\alpha_{k,i} \in \mathbb{N}_0$ is the integer exponent of the $i$-th velocity component
of the $k$-th moment. We write $\boldsymbol{\alpha}_{k} = (\alpha_{k,1},\ldots,\alpha_{k,d})$ for the exponent vector of the $k$-th constrained moment, collect the
constraint monomials into
\begin{equation}
    \mathbf{p}(\mathbf{v}_l) := \left(p_{1}(\mathbf{v}_l),\ldots,p_{M}(\mathbf{v}_l)\right)^{T},\label{eq:p-def}
\end{equation}
and denote by $m_{\boldsymbol{\alpha}}$ the component of $\mathbf{m}$ whose exponent vector is $\boldsymbol{\alpha}$, so that the density is given by $m_{0,\ldots,0}$ and the
momentum components are $m_{\boldsymbol{\epsilon}_{i}}$, with $\boldsymbol{\epsilon}_{i}$ the $i$-th unit multi-index of the form $(0,\ldots,0,1,0,\ldots,0)$.
Similarly, given a set of moment exponents $\beta_{k,i}$ for the higher-order moments required for the flux $\mathbf{F}$, we define
$\mathbf{q}(\mathbf{v}_l) = \left(q_{1}(\mathbf{v}_l),\ldots,q_{M'}(\mathbf{v}_l)\right)^{T}$ with $q_{k}(\mathbf{v}_l) := \prod_{i} v_{l,i}^{\beta_{k,i}}$, and the
corresponding moment measurement matrix $\left(\mathbf{A}_{\mathrm{next}}\right)_{kl} = \Delta v_{l}q_{k}(\mathbf{v}_{l})$.

Note that each row of $\mathbf{A}$ measures a single monomial. In $d>1$ dimensions the specific energy, which corresponds to
$\|\mathbf{v}\|^{2} = \sum_{i=1}^{d}v_{i}^{2}$, is therefore not a single row of $\mathbf{A}$: it is constrained precisely when all $d$ diagonal second-order moments
$m_{2\boldsymbol{\epsilon}_{i}}$, $i = 1,\ldots,d$, are constrained individually. Throughout the analysis below, the phrase ``the energy is among the constrained
moments'' is to be read in this sense.

We find $\mathbf{f}$ via constrained optimization, similar to the classical maximum entropy approach~\cite{levermore1996moment}.
However, we decompose $\mathbf{f}$ into a product of a (known) weighting function $\mathbf{w}$ and an unknown $\mathbf{g}$, i.e.
we write $\mathbf{f} = \mathbf{w} \odot \mathbf{g}$.
The $\mathbf{w}$ can be set to $\mathbf{1}$, recovering the discrete version of the maximum entropy approach, or to a more sophisticated weighting which improves the conditioning of the optimization problem. For example, the weighting function can be chosen as the local Maxwellian computed on the basis of the lower-order moments (density, momentum, energy), and this is the approach considered in the present work, although most of the statements below are valid for any positive-valued weighting function.
This decomposition, first suggested in~\cite{oblapenko2026sparse}, also makes it possible to enforce a sparse structure on $\mathbf{g}$ (and as a consequence, on $\mathbf{f}$), thus allowing for more efficient computation of the moments, fluxes, and more memory-efficient storage of the distribution. Additionally, sparsity also allows for quantitative assessment of whether the discrete support of the distribution can be reduced without significant loss of accuracy.

We find the values $\mathbf{g}$ by solving the following optimization problem~\cite{oblapenko2026sparse}
\begin{equation}
  \begin{aligned}
      \min_{\substack{g_{l}}} \quad & \sum_{l} \Delta v_{l} w_{l} g_{l} \log \left(w_{l} g_{l}\right) + \lambda |\mathbf{g}|_1\\
  \textrm{s.t.}  \quad & \hat{\mathbf{A}} \mathbf{g} = \mathbf{m},\\
  \textrm{s.t.} \quad & g_{l} \ge 0.
  \end{aligned}\label{eq:optimization-spec}
\end{equation}
Here $\hat{\mathbf{A}}$ is the modified moment measurement matrix, with $\hat{\mathbf{A}}_{ij} = \mathbf{A}_{ij} w_j$. It should be the moment measurement matrix $\mathbf{A}$ is defined only by the
set of moments acting as constraints and the velocity space discretization, whereas $\hat{\mathbf{A}}$ may potentially vary with the weights $\mathbf{w}$.

This leads to the following sequence of mappings: $\mathbf{m} \mapsto (\mathbf{w}, \mathbf{g}) \mapsto \mathbf{f} \mapsto \mathbf{m}_{\mathrm{next}}$, where $\mathbf{m}_{\mathrm{next}}$ are the higher-order moments appearing in the flux function. They can be computed as $\mathbf{A}_{\mathrm{next}} \mathbf{f}$.

In the analysis below, we assume the following holds:
\begin{enumerate}
    \item The matrix $\mathbf{A}$ is fixed.\label{assumption:fixed-matrix}
    \item $N>M$ and $\mathrm{rank}(\mathbf{A}) = M$, i.e. the set of constraints is underdetermined and the moment measurement matrix defining the constraints has full row rank.\label{assumption:rank}
    \item Density (which is the zeroth-order moment $m_{0,\ldots,0}$) is included in the set of conserved moments, i.e. the zero multi-index is among the $\boldsymbol{\alpha}_{k}$, so that $\mathbf{A}_{kl} = \Delta v_{l}$ for all $l$ for that row $k$.\label{assumption:density-constraint}
    \item The weighting function $w(\mathbf{v})$ is positive-valued.\label{assumption:positive-weight}
    \item The underlying velocity grid, i.e. the nodes $v_l$ and associated $\Delta v_{l}$ are fixed. \label{assumption:fixed-grid}
\end{enumerate}

\subsection{Existence and uniqueness of solution}
We denote the realizability cone $\mathcal{R}_{M} = \{\mathbf{A}\mathbf{f} | \mathbf{f} \in \mathbb{R}^{N}_{>0} \} \subset \mathbb{R}^{M}$. Under
assumptions~\ref{assumption:fixed-matrix}--\ref{assumption:fixed-grid}, $\mathcal{R}_{M}$ is an open, convex, pointed cone. Indeed, it is a convex
cone as the linear image of the convex cone $\mathbb{R}^{N}_{>0}$; it is pointed because assumption~\ref{assumption:density-constraint} forces
$m_{0,\ldots,0} = \sum_{l}\Delta v_{l}f_{l} > 0$ for every $\mathbf{m}\in\mathcal{R}_{M}$, so that $\mathcal{R}_{M}\cap\left(-\mathcal{R}_{M}\right) = \varnothing$;
and it is open because relative interiors commute with linear images~\cite{rockafellar}, so that
\begin{equation}
    \mathbf{A}\left(\mathbb{R}^{N}_{>0}\right) = \mathbf{A}\left(\mathrm{relint}\mathbb{R}^{N}_{\ge 0}\right) = \mathrm{relint}\mathbf{A}\left(\mathbb{R}^{N}_{\ge 0}\right) = \mathrm{relint}\mathcal{R}_{M},
\end{equation}
where $\mathrm{relint}$ denotes the relative interior of a set. Assumption~\ref{assumption:rank} makes $\mathcal{R}_{M}$ full-dimensional in $\mathbb{R}^{M}$, so that its relative interior coincides
with its interior, and thus $\mathcal{R}_{M}$ is open. The reader is referred to~\cite{hauck2008convex} for an analogous proof in the continuous setting.

\begin{theorem}\label{thm:existence}
    Assumptions~\ref{assumption:fixed-matrix}--\ref{assumption:fixed-grid} hold, and $\mathbf{m} \in \mathcal{R}_{M}$. Then there exists a unique solution to the primal problem~\eqref{eq:optimization-spec}, and the solution is positive.
\end{theorem}
\begin{proof}
    Since $\Delta v_l > 0$ and we assume~\ref{assumption:density-constraint}, the set $\{\mathbf{f} \geq 0: \mathbf{A}\mathbf{f}=\mathbf{m}\}$ is compact. The objective function is strictly convex, and therefore a minimizer $\mathbf{f}^{\ast}$ exists and is unique.
    Now assume that $\exists l: f^{\ast}_l = 0$. We can pick a non-negative feasible point $\mathbf{f}^{\ast\ast} \in \mathcal{R}_{M}$ and consider the parameterization $\mathbf{f}(t) = t\mathbf{f}^{\ast\ast} + (1-t)\mathbf{f}^{\ast},\: t\in[0,1]$. Any of the $\mathbf{f}(t)$ is feasible due to the linearity of the constraints.
    
    The derivative of the objective function with respect to $t$ is given by 
    \begin{equation}
        \sum_{l} \Delta v_{l} ({f}^{\ast\ast}_l - {f}^{\ast}_l) (\log \left(t{f}^{\ast\ast}_l + (1-t){f}^{\ast}_l\right) + 1) + \sum_l \partial_t \left|\frac{{f}(t)_l}{w_l}\right|.
    \end{equation}
    We have that there exists an $l$ such that $\lim_{t \to 0} \log \left({f}^{\ast\ast}_l + (1-t){f}^{\ast}_l\right) = -\infty$, since we assumed $\exists l: f^{\ast}_l = 0$.
    All the other terms in the derivative of the objective function with respect to $t$ are finite, since the weighting function is positive-valued. Therefore, a smaller value of the objective function can be obtained than that for the solution $\mathbf{f}^{\ast}$, since in the feasible vicinity of $\mathbf{f}^{\ast}$ the objective function can be made smaller. This contradicts the uniqueness of the solution, thus $\mathbf{f}^{\ast}$ is always strictly positive.
\end{proof}

\begin{corollary}
    Unlike the continuous case~\cite{levermore1996moment,junk1998domain,hauck2008convex}, the discrete problem always has a unique solution.
\end{corollary}

\begin{corollary}\label{cor:strong-duality}
    Since the primal problem is convex, and the solution is positive-valued, Slater's condition is fulfilled and strong duality holds. In particular the dual optimum is attained.
\end{corollary}

To solve~\eqref{eq:optimization-spec}, we can therefore use the unconstrained dual problem, and for more details refer the reader to~\cite{oblapenko2026sparse}). The dual problem is given by
\begin{equation}
    \min \left(\mathbf{y}^T \mathbf{m}  + \sum_l \Delta v_l z_l \right) =: \min \Psi(\mathbf{y}).\label{eq:dual-targ}
\end{equation}
Here \begin{equation}
    z_l(\mathbf{y}) = \exp\left(-1 - (\lambda + s_l)\kappa_l\right),
\end{equation}
where $\kappa_l = 1 / (\Delta v_{i} w_{l})$ and $s_l$ is the $l$-th component of the vector $\hat{\mathbf{A}}^T \mathbf{y}$. We note that strictly speaking, the dual problem is given by $\max (-\Psi(\mathbf{y}))$, but we will use the formulation in~\eqref{eq:dual-targ} for convenience. We also note that the dual problem objective is, strictly speaking, also dependent on $\mathbf{m}$, i.e. $\Psi(\mathbf{y}) = \Psi(\mathbf{y}; \mathbf{m})$; however, we will drop the dependence on $\mathbf{m}$ in the notation unless needed for clarity.
Having computed the solution $\mathbf{y}$ of the dual problem, we find the primal solution as
\begin{equation}
    g_l = \frac{1}{w_l} z_l(\mathbf{y}).\label{eq:g_i_from_y}
\end{equation}
The dual problem is solved via a Newton method with Armijo backtracking line search, and the reader is referred to~\cite{oblapenko2026sparse} for more details. We note that poor conditioning of the dual problem has been observed in~\cite{schaerer201735}, and the authors proposed using a change of basis to ensure numerical stability at equilibrium. In the present work, the conditioning is improved by row equilibration of the moment measurement matrix $\hat{\mathbf{A}}$ as described in~\cite{oblapenko2026sparse}.

\subsection{Smoothness of the closure}\label{subsec:smoothness}
Using
$\left(\hat{\mathbf{A}}^{T}\mathbf{y}\right)_{l} = \sum_{k} y_{k}\mathbf{A}_{kl}w_{l} = \Delta v_{l}w_{l}\mathbf{y}^{T}\mathbf{p}(\mathbf{v}_{l})$, the
reconstructed distribution function $f_{l} = w_{l}g_{l}$ takes the form
\begin{equation}
    f_{l}(\mathbf{y}) = \tilde{\phi}_{l}\exp\left(-\mathbf{y}^{T}\mathbf{p}(\mathbf{v}_{l})\right),
    \qquad
    \tilde{\phi}_{l} := \exp\left(-1 - \lambda\kappa_{l}\right) > 0.\label{eq:gibbs-form}
\end{equation}
Neither the $L_{1}$ penalty $\lambda |g|_1$ nor the weighting $\mathbf{w}$ therefore changes the functional form of the closure, the feasible set, or the realizability
domain. Two consequences are worth noting. Firstly,
sparsity of the reconstruction is never exact, since $\tilde{\phi}_{l}>0$ forces $f_{l}>0$. Secondly, for $\lambda = 0$ one has
$\tilde{\boldsymbol{\phi}} = e^{-1}$ and the primal solution does not depend on $\mathbf{w}$ at all; the weighting then functions purely as a
preconditioner through $\hat{\mathbf{A}}$, the initialization, and the row scaling.

We introduce the Gram matrices computed on the basis of the solution $\mathbf{f}^{\ast}$ to the optimization problem~(\ref{eq:optimization-spec}):
\begin{equation}
    \mathcal{H} := \left\langle \mathbf{p}\mathbf{p}^{T}\right\rangle,\qquad
    \mathcal{B} := \left\langle \mathbf{q}\mathbf{p}^{T}\right\rangle,\qquad
    \mathcal{K}^{(i)} := \left\langle v_{i}\mathbf{p}\mathbf{p}^{T}\right\rangle,\quad i = 1,\ldots,d,\label{eq:gram-matrices}
\end{equation}
where $v_{i}$ denotes the $i$-th component of the integration variable, that is, $\mathcal{K}^{(i)}$ is defined by the discrete integration of the $i$-th component of all the discrete velocity nodes. For one-dimensional distributions we drop the superscript and write
$\mathcal{K} := \mathcal{K}^{(1)}$.
From~\eqref{eq:A-def} we have that $\mathcal{H} = \mathbf{A}\mathrm{diag}\left(f^{\ast}_{l}/\Delta v_{l}\right)\mathbf{A}^{T}$, so Theorem~\ref{thm:existence} together with
assumption~\ref{assumption:rank} gives
\begin{equation}
    \mathcal{H} = \mathcal{H}^{T} \succ 0 \qquad \text{for every } \mathbf{m}\in\mathcal{R}_{M},\label{eq:H-spd}
\end{equation}
where $\succ 0$ denotes that the matrix is positive definite.

\begin{theorem}\label{thm:analyticity}
    Let assumptions~\ref{assumption:fixed-matrix}--\ref{assumption:fixed-grid} hold and let $\mathbf{w}$ be fixed. Let $\mathbf{y}^{\ast}$ denote the solution of the dual problem. Then the mappings
    \begin{equation*}
        \mathbf{m} \mapsto \mathbf{y}^{\ast}(\mathbf{m}),\qquad
        \mathbf{m} \mapsto \mathbf{f}^{\ast}(\mathbf{m}),\qquad
        \mathbf{m} \mapsto \mathbf{m}_{\mathrm{next}} = \mathbf{A}_{\mathrm{next}}\mathbf{f}^{\ast}(\mathbf{m})
    \end{equation*}
    are real-analytic on $\mathcal{R}_{M}$, and
    \begin{equation}
        \frac{\partial \mathbf{y}^{\ast}}{\partial \mathbf{m}} = -\mathcal{H}^{-1},\qquad
        \frac{\partial \mathbf{m}_{\mathrm{next}}}{\partial \mathbf{m}} = \mathcal{B}\mathcal{H}^{-1}.\label{eq:closure-jacobian}
    \end{equation}
\end{theorem}
\begin{proof}
    We have the convex dual objective 
    \begin{equation}
        \Psi(\mathbf{y};\mathbf{m}) = \mathbf{y}^{T}\mathbf{m} + Z(\mathbf{y}),\qquad
        Z(\mathbf{y}) := \sum_{l}\Delta v_{l}\tilde{\phi}_{l}\exp\left(-\mathbf{y}^{T}\mathbf{p}(\mathbf{v}_{l})\right),\label{eq:dual-convex}
    \end{equation}
    where~\eqref{eq:gibbs-form} was used to identify $\sum_{l}\Delta v_{l}z_{l} = Z(\mathbf{y})$.

    The function $Z$ is a finite sum of exponentials of affine functions of $\mathbf{y}$ and is therefore real-analytic on all of $\mathbb{R}^{M}$, with
    \begin{equation}
        \nabla Z(\mathbf{y}) = -\mathbf{A}\mathbf{f}(\mathbf{y}),\qquad
        \nabla^{2}Z(\mathbf{y}) = \sum_{l}\Delta v_{l}f_{l}(\mathbf{y})\mathbf{p}(\mathbf{v}_{l})\mathbf{p}(\mathbf{v}_{l})^{T}
        = \mathbf{A}\mathrm{diag}\left(f_{l}(\mathbf{y})/\Delta v_{l}\right)\mathbf{A}^{T}.
    \end{equation}
    Since $f_{l}(\mathbf{y})>0$ for every finite $\mathbf{y}$ by~\eqref{eq:gibbs-form}, assumption~\ref{assumption:rank} yields $\nabla^{2}Z \succ 0$ everywhere. Therefore
    $\Psi$ is strictly convex in $\mathbf{y}$ and has at most one stationary point, which is its unique global minimizer; by Corollary~\ref{cor:strong-duality} it is
    attained, so $\mathbf{y}^{\ast}(\mathbf{m})$ is well defined and single-valued on $\mathcal{R}_{M}$; the dual problem~\eqref{eq:dual-targ} thus has a unique
    solution, complementing the primal statement of Theorem~\ref{thm:existence}.

    Consider the mapping arising from the stationarity condition:
    \begin{equation}
        G(\mathbf{y},\mathbf{m}) := \nabla_{\mathbf{y}}\Psi(\mathbf{y};\mathbf{m}) = \mathbf{m} - \mathbf{A}\mathbf{f}(\mathbf{y}),
    \end{equation}
    which is jointly real-analytic on $\mathbb{R}^{M}\times\mathcal{R}_{M}$ and $G\left(\mathbf{y}^{\ast}(\mathbf{m}),\mathbf{m}\right)=0$. Its partial
    derivatives there are
    \begin{equation}
        \frac{\partial G}{\partial \mathbf{y}} = \nabla^{2}Z\left(\mathbf{y}^{\ast}\right) = \mathcal{H},\qquad
        \frac{\partial G}{\partial \mathbf{m}} = \mathbf{I},
    \end{equation}
    and $\mathcal{H}$ is invertible by~\eqref{eq:H-spd}. The implicit function theorem  therefore provides a real-analytic local
    solution $\mathbf{m}\mapsto\mathbf{y}^{\ast}(\mathbf{m})$ with
    \begin{equation}
        \frac{\partial \mathbf{y}^{\ast}}{\partial \mathbf{m}}
        = -\left(\frac{\partial G}{\partial \mathbf{y}}\right)^{-1}\frac{\partial G}{\partial \mathbf{m}} = -\mathcal{H}^{-1}.
    \end{equation}
    Since the dual minimizer is unique, the local solution coincides with $\mathbf{y}^{\ast}$ on its domain; as $\mathbf{m}\in\mathcal{R}_{M}$ was arbitrary and
    $\mathcal{R}_{M}$ is open, $\mathbf{y}^{\ast}$ is real-analytic on all of $\mathcal{R}_{M}$.

    Real-analyticity of $\mathbf{m}\mapsto\mathbf{f}^{\ast}(\mathbf{m}) = \mathbf{f}\left(\mathbf{y}^{\ast}(\mathbf{m})\right)$ follows since the composition of
    real-analytic mappings is real-analytic, and differentiating~\eqref{eq:gibbs-form} through the chain rule gives
    \begin{equation}
        \frac{\partial f^{\ast}_{l}}{\partial \mathbf{m}}
        = -f^{\ast}_{l}\mathbf{p}(\mathbf{v}_{l})^{T}\frac{\partial \mathbf{y}^{\ast}}{\partial\mathbf{m}}
        = f^{\ast}_{l}\mathbf{p}(\mathbf{v}_{l})^{T}\mathcal{H}^{-1}.\label{eq:df-dm}
    \end{equation}
    Finally $\mathbf{m}_{\mathrm{next}} = \mathbf{A}_{\mathrm{next}}\mathbf{f}^{\ast}$ is linear in $\mathbf{f}^{\ast}$, and thus also real-analytic, and
    \begin{equation}
        \frac{\partial \mathbf{m}_{\mathrm{next}}}{\partial \mathbf{m}}
        = \sum_{l}\Delta v_{l}\mathbf{q}(\mathbf{v}_{l})f^{\ast}_{l}\mathbf{p}(\mathbf{v}_{l})^{T}\mathcal{H}^{-1}
        = \left\langle \mathbf{q}\mathbf{p}^{T}\right\rangle\mathcal{H}^{-1} = \mathcal{B}\mathcal{H}^{-1}.
    \end{equation}
\end{proof}

\begin{corollary}\label{cor:weight-analytic}
    Let the density, the momentum and the energy be among the constrained moments, and let $\mathbf{w}$ be the shifted Maxwellian fitted to the state,
    \begin{equation}
        w_{l}(\mathbf{m}) = \exp\left(-\beta\left\|\mathbf{v}_{l}-\bar{\mathbf{v}}\right\|^{2}\right),\label{eq:maxwellian-weight}
    \end{equation}
    with $\bar{\mathbf{v}}(\mathbf{m})$ the mean velocity and $\beta = 1/T>0$ fixed by matching the grid-truncated specific energy, i.e. by
    $E\left(\beta,\bar{\mathbf{v}}\right) = e(\mathbf{m})$, where
    \begin{equation}
        \begin{aligned}
            E\left(\beta,\bar{\mathbf{v}}\right) &:= \frac{\sum_{l}\Delta v_{l}\left\|\mathbf{v}_{l}-\bar{\mathbf{v}}\right\|^{2}e^{-\beta\|\mathbf{v}_{l}-\bar{\mathbf{v}}\|^{2}}}
            {\sum_{l}\Delta v_{l}e^{-\beta\|\mathbf{v}_{l}-\bar{\mathbf{v}}\|^{2}}},\\
            e(\mathbf{m}) &:= \sum_{i=1}^{d}\left(\frac{m_{2\boldsymbol{\epsilon}_{i}}}{m_{0,\ldots,0}} - \bar{v}_{i}^{2}\right),\qquad
            \bar{v}_{i} := \frac{m_{\boldsymbol{\epsilon}_{i}}}{m_{0,\ldots,0}},\label{eq:maxwellian-exact-constraints}
        \end{aligned}
    \end{equation}
    Then $\mathbf{m}\mapsto\mathbf{w}(\mathbf{m})$ is real-analytic on the open set
    \begin{equation}
        \mathcal{R}^{w}_{M} := \left\{\mathbf{m}\in\mathcal{R}_{M}:\ \min_{l}\left\|\mathbf{v}_{l}-\bar{\mathbf{v}}(\mathbf{m})\right\|^{2}
        < e(\mathbf{m}) < \frac{\sum_{l}\Delta v_{l}\left\|\mathbf{v}_{l}-\bar{\mathbf{v}}(\mathbf{m})\right\|^{2}}{\sum_{l}\Delta v_{l}}\right\},\label{eq:Rw}
    \end{equation}
    and the closure $\mathbf{m}\mapsto\mathbf{m}_{\mathrm{next}}$ remains real-analytic on $\mathcal{R}^{w}_{M}$.
\end{corollary}
\begin{proof}
    On $\mathcal{R}_{M}$ one has $m_{0,\ldots,0}>0$, so $\bar{\mathbf{v}}$ and $e$ are rational functions of $\mathbf{m}$ with non-vanishing denominator and are
    thus real-analytic there.

    Fix $\bar{\mathbf{v}}\in\mathbb{R}^{d}$ and let $\mu_{\beta}$ be the probability measure on the grid with weights proportional to
    $\Delta v_{l}e^{-\beta\|\mathbf{v}_{l}-\bar{\mathbf{v}}\|^{2}}$, so that $E(\beta,\bar{\mathbf{v}}) = \left\langle \|\mathbf{v}-\bar{\mathbf{v}}\|^{2}\right\rangle_{\mu_{\beta}}$.
    As a ratio of finite sums of exponentials with strictly positive denominator, $E$ is jointly real-analytic in $(\beta,\bar{\mathbf{v}})$, and a direct
    computation gives the standard exponential-family identity
    \begin{equation}
        \frac{\partial E}{\partial \beta} = -\mathrm{Var}_{\mu_{\beta}}\left(\left\|\mathbf{v}-\bar{\mathbf{v}}\right\|^{2}\right) \le 0,
    \end{equation}
    where $\mathrm{Var}_{\mu_{\beta}}$ denotes the variance computed with respect to the measure $\mu_{\beta}$. Equality above is possible if and only if the numbers $\|\mathbf{v}_{l}-\bar{\mathbf{v}}\|^{2}$ are all equal.  The latter cannot happen: if
    $\|\mathbf{v}_{l}-\bar{\mathbf{v}}\|^{2} = r$ for all $l$, then the polynomial $\|\mathbf{v}\|^{2} - 2\bar{\mathbf{v}}^{T}\mathbf{v} + \|\bar{\mathbf{v}}\|^{2} - r$
    vanishes on the entire grid. By hypothesis the multi-indices $\mathbf{0}$, $\boldsymbol{\epsilon}_{i}$ and $2\boldsymbol{\epsilon}_{i}$, $i=1,\ldots,d$, are all
    among the $\boldsymbol{\alpha}_{k}$, so the constant, the linear monomials and $\|\mathbf{v}\|^{2} = \sum_{i}v_{i}^{2}$ all lie in the span of
    $p_{1},\ldots,p_{M}$; the above combination of them is non-zero since the coefficient of each $v_{i}^{2}$ equals one. This means that there exists a non-zero
    vector $\mathbf{c}$ with $\mathbf{c}^{T}\mathbf{p}(\mathbf{v}_{l}) = 0$ for all $l$, therefore $\mathbf{c}^{T}\mathbf{A} = 0$ by~\eqref{eq:A-def}, contradicting
    assumption~\ref{assumption:rank}. Therefore $\partial E/\partial\beta < 0$ strictly.

    Consequently $\beta \mapsto E(\beta,\bar{\mathbf{v}})$ is a real-analytic strictly decreasing bijection of $(0,\infty)$ onto the open interval appearing
    in~\eqref{eq:Rw}: as $\beta\to 0^{+}$ the measure $\mu_{\beta}$ tends to the $\Delta v$-weighted uniform measure on the grid, giving the right endpoint, and as
    $\beta\to\infty$ it concentrates on the nodes closest to $\bar{\mathbf{v}}$, giving the left endpoint. Applying the analytic implicit function theorem to
    $\left(\beta,\bar{\mathbf{v}},e\right)\mapsto E(\beta,\bar{\mathbf{v}}) - e$, whose $\beta$-derivative is non-zero, shows that
    $\left(\bar{\mathbf{v}},e\right)\mapsto\beta$ is real-analytic wherever $e$ lies in that interval; composing with the real-analytic
    $\mathbf{m}\mapsto\left(\bar{\mathbf{v}}(\mathbf{m}),e(\mathbf{m})\right)$ gives real-analyticity of $\beta(\mathbf{m})$, and therefore of $\mathbf{w}(\mathbf{m})$
    by~\eqref{eq:maxwellian-weight}. The set $\mathcal{R}^{w}_{M}$ is open, being the intersection of the open set $\mathcal{R}_{M}$ with the preimage of an open
    condition under continuous mappings.

    Finally, $\tilde{\boldsymbol{\phi}}$ in~\eqref{eq:gibbs-form} now depends real-analytically on $\mathbf{m}$, so $G(\mathbf{y},\mathbf{m})$ in the proof of
    Theorem~\ref{thm:analyticity} remains jointly real-analytic; its $\mathbf{y}$-derivative is still $\mathcal{H}\succ 0$, since the $\mathbf{m}$-dependence of
    $\tilde{\boldsymbol{\phi}}$ does not enter it. The implicit function theorem applies verbatim and $\mathbf{m}\mapsto\mathbf{m}_{\mathrm{next}}$ is real-analytic
    on $\mathcal{R}^{w}_{M}$; only the Jacobian~\eqref{eq:closure-jacobian} acquires an extra term, computed in Corollary~\ref{cor:jacobian-varying-w} below.
\end{proof}

\begin{remark}
    We note that the lower bound on temperature coincides with that given by~\cite[Prop.~4.1]{mieussens2000discrete}; the upper bound is different, as we assume a non-negative temperature, whereas in~\cite{mieussens2000discrete}, no restriction on the temperature is imposed (see Remark 4.1 therein).
\end{remark}

With a state-dependent weighting the closure thus remains real-analytic, but its derivatives acquire correction terms relative to~\eqref{eq:closure-jacobian}. Since
these derivatives are the basis of the hyperbolicity analysis of Section~\ref{subsec:hyperbolicity}, we record them here in the general multi-dimensional setting.

\begin{corollary}\label{cor:jacobian-varying-w}
    Let the hypotheses of Corollary~\ref{cor:weight-analytic} hold, let $\mathbf{m}\in\mathcal{R}^{w}_{M}$, and define
    \begin{equation}
        \mathcal{G} := \sum_{l}\Delta v_{l}f^{\ast}_{l}\,\mathbf{p}(\mathbf{v}_{l})\left(\frac{\partial \kappa_{l}}{\partial\mathbf{m}}\right)^{T}.\label{eq:G-def}
    \end{equation}
    Then
    \begin{equation}
        -\frac{\partial\mathbf{y}^{\ast}}{\partial\mathbf{m}} = \mathcal{H}^{-1}\left(\mathbf{I}+\lambda\mathcal{G}\right),\qquad
        \frac{\partial f^{\ast}_{l}}{\partial\mathbf{m}}
        = f^{\ast}_{l}\left[\mathbf{p}(\mathbf{v}_{l})^{T}\mathcal{H}^{-1}\left(\mathbf{I}+\lambda\mathcal{G}\right) - \lambda\left(\frac{\partial\kappa_{l}}{\partial\mathbf{m}}\right)^{T}\right],\label{eq:dy-perturbed}
    \end{equation}
    and consequently, for any vector-valued function $\mathbf{r}$ on the grid nodes,
    \begin{equation}
        \frac{\partial \left\langle\mathbf{r}\right\rangle}{\partial\mathbf{m}}
        = \left\langle\mathbf{r}\mathbf{p}^{T}\right\rangle\mathcal{H}^{-1}\left(\mathbf{I}+\lambda\mathcal{G}\right)
        - \lambda\sum_{l}\Delta v_{l}f^{\ast}_{l}\,\mathbf{r}(\mathbf{v}_{l})\left(\frac{\partial\kappa_{l}}{\partial\mathbf{m}}\right)^{T}.\label{eq:moment-map-varying-w}
    \end{equation}
    In particular, taking $\mathbf{r} = \mathbf{q}$,
    \begin{equation}
        \frac{\partial \mathbf{m}_{\mathrm{next}}}{\partial\mathbf{m}}
        = \mathcal{B}\mathcal{H}^{-1}\left(\mathbf{I}+\lambda\mathcal{G}\right)
        - \lambda\sum_{l}\Delta v_{l}f^{\ast}_{l}\,\mathbf{q}(\mathbf{v}_{l})\left(\frac{\partial\kappa_{l}}{\partial\mathbf{m}}\right)^{T},\label{eq:closure-jacobian-varying-w}
    \end{equation}
    which reduces to~\eqref{eq:closure-jacobian} whenever $\lambda\mathcal{G}$ vanishes. For $\lambda = 0$ the closure does not depend on $\mathbf{w}$ at all, every
    correction term vanishes, and Theorem~\ref{thm:analyticity} applies verbatim.
\end{corollary}
\begin{proof}
    By Corollary~\ref{cor:weight-analytic} the mappings $\mathbf{m}\mapsto\kappa_{l}(\mathbf{m})$ are real-analytic on $\mathcal{R}^{w}_{M}$, and taking the logarithm
    of~\eqref{eq:gibbs-form} we get
    \begin{equation}
        \log f^{\ast}_{l} = - 1 - \lambda\kappa_{l}(\mathbf{m}) - \left(\mathbf{y}^{\ast}\right)^{T}\mathbf{p}(\mathbf{v}_{l}),
    \end{equation}
    so that
    \begin{equation}
        \frac{\partial f^{\ast}_{l}}{\partial\mathbf{m}}
        = f^{\ast}_{l}\left[-\mathbf{p}(\mathbf{v}_{l})^{T}\frac{\partial\mathbf{y}^{\ast}}{\partial\mathbf{m}} - \lambda\left(\frac{\partial\kappa_{l}}{\partial\mathbf{m}}\right)^{T}\right].\label{eq:df-varying-raw}
    \end{equation}
    Differentiating the constraint $\sum_{l}\Delta v_{l}\mathbf{p}(\mathbf{v}_{l})f^{\ast}_{l} = \mathbf{m}$ with respect to $\mathbf{m}$ and
    inserting~\eqref{eq:df-varying-raw} gives $\mathbf{I} = -\mathcal{H}\,\partial\mathbf{y}^{\ast}/\partial\mathbf{m} - \lambda\mathcal{G}$, which is the first identity
    in~\eqref{eq:dy-perturbed}; substituting it back into~\eqref{eq:df-varying-raw} gives the second. Multiplying the second identity by
    $\Delta v_{l}\mathbf{r}(\mathbf{v}_{l})$ and summing over $l$ yields~\eqref{eq:moment-map-varying-w}. Finally, for $\lambda = 0$ we have
    $\tilde{\boldsymbol{\phi}} = e^{-1}$ in~\eqref{eq:gibbs-form}, so $\mathbf{f}^{\ast}$ is independent of $\mathbf{w}$, and every $\lambda$-proportional term above
    vanishes.
\end{proof}

\subsection{Hyperbolicity of the closed system in one dimension}\label{subsec:hyperbolicity}

We now specialize to $d = 1$, with distinct nodes $v_{1} < v_{2} < \ldots < v_{N}$, constraint monomials $p_{k}(v) = v^{k-1}$, $k=1,\ldots,M$ (i.e. $\alpha_{k,1} = k-1$), and a single
predicted moment $\mathbf{q}(v) = v^{M}$. The state vector is $\mathbf{m} = (m_{1},\ldots,m_{M})^{T}$ with $m_{k} = \left\langle v^{k-1}\right\rangle$, and the
closed moment system is
\begin{equation}
    \partial_{t}\mathbf{m} + \partial_{x}\mathbf{F}(\mathbf{m}) = 0,\label{eq:moment-system-1d}
\end{equation}
with the flux given by
\begin{equation}
    \mathbf{F}(\mathbf{m}) = \left(m_{2},\ldots,m_{M},\ \mathbf{A}_{\mathrm{next}}\mathbf{f}^{\ast}(\mathbf{m})\right)^{T}
    = \sum_{l}\Delta v_{l}v_{l}\mathbf{p}(v_{l})f^{\ast}_{l}.\label{eq:moment-system-1d-closure}
\end{equation}
The second expression for $\mathbf{F}$ holds because $vp_{k} = p_{k+1}$ for $k\le M-1$ and $vp_{M} = q$.

The closure~\eqref{eq:optimization-spec} is an instance of ADVM~\cite{pichard2025entropy}. Indeed, writing $\mathsf{m}_{l} := \Delta v_{l}f_{l}$ for the mass carried by node $l$, the objective
of~\eqref{eq:optimization-spec} reads
\begin{equation}
    \sum_{l}\eta_{l}(\mathsf{m}_{l}),\qquad \eta_{l}(\mathsf{m}_l) := \mathsf{m}_l\log \mathsf{m}_l - c_{l}\mathsf{m}_l,\qquad
    c_{l} := \log \Delta v_{l} - \lambda\kappa_{l},\label{eq:advm-identification}
\end{equation}
a sum of node-wise strictly convex functions minimized under the linear moment constraints $\mathbf{A}\mathbf{f} = \mathbf{m}$, with more nodes than moments
($N>M$). Both the $L_{1}$ penalty and the weighting $\mathbf{w}$ enter only through the linear coefficients $c_{l}$, in accordance
with~\eqref{eq:gibbs-form}. Part~(i) of the following theorem is therefore a specialization of~\cite[Prop.~4.2]{pichard2025entropy}; we record it in the present
notation because part~(ii), which is new, relies on the identity $\mathbf{F}' = \mathcal{K}\mathcal{H}^{-1}$.

\begin{theorem}\label{thm:hyperbolicity-fixed-w}
    Let $d=1$, let $\mathbf{w}$ be fixed, let assumptions~\ref{assumption:fixed-matrix}--\ref{assumption:fixed-grid} hold, and let
    $\mathbf{m}\in\mathcal{R}_{M}$. Define $h_l(f_l) := f_l\log(f_l) + \lambda\kappa_{l}f_l$ and
    \begin{equation}
        \eta(\mathbf{m}) := \sum_{l}\Delta v_{l}h_l\left(f^{\ast}_{l}(\mathbf{m})\right),\qquad
        \psi(\mathbf{m}) := \sum_{l}\Delta v_{l}v_{l}h_l\left(f^{\ast}_{l}(\mathbf{m})\right).\label{eq:entropy-pair}
    \end{equation}
    Then:
    \begin{enumerate}
        \item $\eta$ is strictly convex on $\mathcal{R}_{M}$, with $\nabla\eta = -\mathbf{y}^{\ast}$ and $\nabla^{2}\eta = \mathcal{H}^{-1}\succ0$; the pair
              $(\eta,\psi)$ is an entropy/entropy-flux pair for~\eqref{eq:moment-system-1d}--\eqref{eq:moment-system-1d-closure}; and
              $\mathbf{F}'(\mathbf{m}) = \mathcal{K}\mathcal{H}^{-1}$, so that the system is symmetric hyperbolic, $\mathcal{H}^{-1}$ being a
              symmetrizer~\cite{godlewski2013numerical};
        \item the characteristic speeds of~\eqref{eq:moment-system-1d}--\eqref{eq:moment-system-1d-closure} are the $M$ nodes $\theta_{1},\ldots,\theta_{M}$ of the $M$-point Gauss quadrature rule associated
              with the reconstructed measure $\mu^{\ast}$. They are real and \emph{pairwise distinct}, so that the system is strictly hyperbolic at every
              $\mathbf{m}\in\mathcal{R}_{M}$, and they satisfy the state-independent bound
              \begin{equation}
                  v_{1} < \theta_{1} < \theta_{2} < \ldots < \theta_{M} < v_{N}.\label{eq:speed-bound}
              \end{equation}
    \end{enumerate}
\end{theorem}
\begin{proof}
    \emph{(i)} By construction $\eta$ is the optimal value of the primal objective in~\eqref{eq:optimization-spec}, so strong duality
    (Corollary~\ref{cor:strong-duality}) gives $\eta(\mathbf{m}) = -\Psi\left(\mathbf{y}^{\ast}(\mathbf{m})\right)$. Since $\mathbf{y}^{\ast}$ is a stationary point
    of $\Psi(\mathbf{y};\mathbf{m})$, the envelope theorem gives $\nabla\eta = -\mathbf{y}^{\ast}$, and Theorem~\ref{thm:analyticity} then yields
    $\nabla^{2}\eta = -\partial\mathbf{y}^{\ast}/\partial\mathbf{m} = \mathcal{H}^{-1}\succ0$, so $\eta$ is strictly convex. Next, from~\eqref{eq:gibbs-form} we have that
    $\partial f^{\ast}_{l}/\partial\mathbf{y} = -f^{\ast}_{l}\mathbf{p}(v_{l})^{T}$, so that in the entropy variables $-\mathbf{y}$ we can write
    \begin{equation}
        \frac{\partial \mathbf{m}}{\partial (-\mathbf{y})} = \sum_{l}\Delta v_{l}\mathbf{p}(v_{l})f^{\ast}_{l}\mathbf{p}(v_{l})^{T} = \mathcal{H},\qquad
        \frac{\partial \mathbf{F}}{\partial (-\mathbf{y})} = \sum_{l}\Delta v_{l}v_{l}\mathbf{p}(v_{l})f^{\ast}_{l}\mathbf{p}(v_{l})^{T} = \mathcal{K}.
    \end{equation}
    Both $\mathcal{H}$ and $\mathcal{K}$ are symmetric, therefore $\mathbf{F}'(\mathbf{m}) = \mathcal{K}\mathcal{H}^{-1}$ by the chain rule and
    $\nabla^{2}\eta\mathbf{F}' = \mathcal{H}^{-1}\mathcal{K}\mathcal{H}^{-1}$ is symmetric with $\nabla^{2}\eta\succ0$: the Friedrichs condition holds, and local
    well-posedness of the Cauchy problem follows in the standard way~\cite{godlewski2013numerical}. It remains to verify entropy compatibility, i.e.
    $\nabla\psi = \mathbf{F}'^{T}\nabla\eta$. Since $h_l'(f) = \log f + 1 + \lambda\kappa_{l}$, taking the logarithm of~\eqref{eq:gibbs-form} gives the primal
    stationarity relation
    \begin{equation}
        h_l'(f^{\ast}_{l}) = \log\tilde{\phi}_{l} + 1 + \lambda\kappa_{l} - \left(\mathbf{y}^{\ast}\right)^{T}\mathbf{p}(v_{l})
        = -\left(\mathbf{y}^{\ast}\right)^{T}\mathbf{p}(v_{l}),\label{eq:primal-stationarity}
    \end{equation}
    the term $1 + \lambda\kappa_{l}$ being cancelled exactly by $\log\tilde{\phi}_{l} = -1-\lambda\kappa_{l}$. Differentiating $\psi$
    in~\eqref{eq:entropy-pair} through the chain rule and inserting first~\eqref{eq:primal-stationarity} and then~\eqref{eq:df-dm} yields
    \begin{equation}
        \frac{\partial\psi}{\partial\mathbf{m}}
        = \sum_{l}\Delta v_{l}v_{l}\,h_l'\left(f^{\ast}_{l}\right)\frac{\partial f^{\ast}_{l}}{\partial\mathbf{m}}
        = -\left(\mathbf{y}^{\ast}\right)^{T}\sum_{l}\Delta v_{l}v_{l}f^{\ast}_{l}\,\mathbf{p}(v_{l})\mathbf{p}(v_{l})^{T}\mathcal{H}^{-1}
        = -\left(\mathbf{y}^{\ast}\right)^{T}\mathcal{K}\mathcal{H}^{-1},
    \end{equation}
    the sum being $\mathcal{K}$ by~\eqref{eq:gram-matrices}. Since $\nabla\eta = -\mathbf{y}^{\ast}$ and $\mathbf{F}' = \mathcal{K}\mathcal{H}^{-1}$, this is
    precisely $\left(\nabla\eta\right)^{T}\mathbf{F}'$, and transposing gives $\nabla\psi = \mathbf{F}'^{T}\nabla\eta$; see
    also~{schaerer2017efficient,schaerer201735,pichard2025entropy}.

    \emph{(ii)} By Theorem~\ref{thm:existence}, $f^{\ast}_{l}>0$ for every $l$, so $\mu^{\ast}$ is a positive measure supported on all $N$ grid
    nodes; by assumption~\ref{assumption:rank}, $N \ge M+1$. Thus orthonormal polynomials $\pi_{0},\pi_{1},\ldots,\pi_{M}$ of $\mu^{\ast}$ exist, with
    $\deg\pi_{k}=k$ and positive leading coefficients $\gamma_{k}$. Let $Q\in\mathbb{R}^{M\times M}$ be the (upper triangular, invertible) change of basis defined by
    $\boldsymbol{\pi} := \left(\pi_{0},\ldots,\pi_{M-1}\right)^{T} = Q^{T}\mathbf{p}$. Orthonormality reads
    \begin{equation}
        Q^{T}\mathcal{H}Q = \left\langle\boldsymbol{\pi}\boldsymbol{\pi}^{T}\right\rangle = \mathbf{I},
    \end{equation}
    or equivalently, $\mathcal{H}^{-1} = QQ^{T}$.
    The three-term recurrence $v\pi_{k} = b_{k+1}\pi_{k+1} + a_{k}\pi_{k} + b_{k}\pi_{k-1}$ with $b_{k} = \gamma_{k-1}/\gamma_{k}>0$, together with
    $\left\langle v\pi_{j}\pi_{k}\right\rangle = 0$ for $|j-k|>1$, gives the Jacobi matrix of $\mu^{\ast}$
    \begin{equation}
        J := Q^{T}\mathcal{K}Q = \left\langle v\boldsymbol{\pi}\boldsymbol{\pi}^{T}\right\rangle,
    \end{equation}
    with
    \begin{equation}
        J_{kk} = a_{k-1},\quad J_{k,k+1} = J_{k+1,k} = b_{k},\quad J_{jk} = 0 \ \text{ for } |j-k|>1.
    \end{equation}
    We therefore have
    \begin{equation}
        \mathbf{F}'(\mathbf{m}) = \mathcal{K}\mathcal{H}^{-1} = \mathcal{K}QQ^{T} = Q^{-T}\left(Q^{T}\mathcal{K}Q\right)Q^{T} = Q^{-T}JQ^{T},
    \end{equation}
    so the flux Jacobian is similar to the Jacobi matrix $J$ and has the same spectrum.

    The matrix $J$ is symmetric tridiagonal and unreduced, since all off-diagonal entries $b_{1},\ldots,b_{M-1}$ are strictly positive. Its eigenvalues are real due to symmetry, and simple because the three-term recurrence determines each eigenvector uniquely up to scale (thus the geometric multiplicity of each eigenvalue is one), and since the matrix is symmetric,
    the geometric and algebraic multiplicities coincide.
    By the Golub--Welsch construction~\cite{golub1969calculation,gautschi2004orthogonal}, the eigenvalues of
    $J$ are precisely the nodes of the $M$-point Gauss quadrature rule for $\mu^{\ast}$, i.e. the zeros of $\pi_{M}$. Zeros of orthogonal polynomials are simple and
    lie strictly inside the convex hull of the support of the measure~\cite{gautschi2004orthogonal}, which here is $\left[v_{1},v_{N}\right]$; this gives~\eqref{eq:speed-bound} and completes the
    proof.
\end{proof}

Part~(i) is~\cite[Prop.~4.2]{pichard2025entropy} under the identification~\eqref{eq:advm-identification}. In case the per-node entropies correspond to a quadrature rule applied to the Boltzmann entropy, the closure coincides with the quadrature-discretised maximum-entropy closure, for which the same symmetric-hyperbolic structure was established in~\cite{schaerer201735,schaerer2017efficient}. The closed bound
$\mathrm{Sp}(\mathbf{F}')\subset[v_{1},v_{N}]$ given in~\cite{pichard2025entropy} is derived by an argument independent of part~(ii): since
$\mathcal{K}-v_{1}\mathcal{H} = \sum_{l}\Delta v_{l}(v_{l}-v_{1})f^{\ast}_{l}\mathbf{p}(v_{l})\mathbf{p}(v_{l})^{T}\succeq 0$ (i.e. is positive semi-definite), the matrix
$\mathbf{F}'-v_{1}\mathbf{I} = \left(\mathcal{K}-v_{1}\mathcal{H}\right)\mathcal{H}^{-1}$ is similar to
$\mathcal{H}^{-1/2}\left(\mathcal{K}-v_{1}\mathcal{H}\right)\mathcal{H}^{-1/2}\succeq 0$ and therefore has non-negative spectrum; a similar arguments applies to
$v_{N}\mathbf{I}-\mathbf{F}'$. Part~(ii) is stronger in two respects: it identifies the characteristic speeds, and it gives \emph{strict} hyperbolicity, the
eigenvalues being pairwise distinct rather than simply real, together with the strict inequalities
in~\eqref{eq:speed-bound}.

It is worth emphasizing that from Theorem~\ref{thm:hyperbolicity-fixed-w} it follows that the bound~\eqref{eq:speed-bound} gives
$\max_{i}\left|\theta_{i}\right| < \max_{l}\left|v_{l}\right|$ \emph{independently of the state}, i.e. it is an a priori CFL bound determined by the velocity grid alone, a property the continuous maximum-entropy closure does not possess.

\begin{theorem}\label{thm:hyperbolicity-varying-w}
    Let the hypotheses of Theorem~\ref{thm:hyperbolicity-fixed-w} hold, except that $\mathbf{w} = \mathbf{w}(\mathbf{m})$ is refitted from the state as in
    Corollary~\ref{cor:weight-analytic}, and let $\mathbf{m}\in\mathcal{R}^{w}_{M}$. Let $\mathcal{G}$ be as in~\eqref{eq:G-def} and denote
    \begin{equation}
        \mathcal{G}_{v} := \sum_{l}\Delta v_{l}v_{l}f^{\ast}_{l}\mathbf{p}(v_{l})\left(\frac{\partial \kappa_{l}}{\partial\mathbf{m}}\right)^{T}.
    \end{equation}
    Then the flux Jacobian is
    \begin{equation}
        \mathbf{F}'(\mathbf{m}) = \mathcal{K}\mathcal{H}^{-1} + \lambda R,\qquad
        R := \mathcal{K}\mathcal{H}^{-1}\mathcal{G} - \mathcal{G}_{v},\label{eq:perturbed-jacobian}
    \end{equation}
    which is in general no longer symmetrized by $\mathcal{H}^{-1}$. Nevertheless the system remains strictly hyperbolic, with characteristic speeds
    $\tilde{\theta}_{1},\ldots,\tilde{\theta}_{M}$ real, pairwise distinct and satisfying
    $\left|\tilde{\theta}_{i}-\theta_{i}\right| \le \lambda\left\|\mathcal{H}^{-1/2}R\mathcal{H}^{1/2}\right\|_{2}$, provided that
    \begin{equation}
        \lambda\left\|\mathcal{H}^{-1/2}R\mathcal{H}^{1/2}\right\|_{2} < \frac{1}{2}\min_{i\ne j}\left|\theta_{i}-\theta_{j}\right|,\label{eq:bauer-fike-condition}
    \end{equation}
    where $\theta_{i}$ are the Gauss nodes of Theorem~\ref{thm:hyperbolicity-fixed-w}. In particular, for $\lambda = 0$ the closure does not depend on $\mathbf{w}$ and
    Theorem~\ref{thm:hyperbolicity-fixed-w} applies verbatim.
\end{theorem}
\begin{proof}
    Corollary~\ref{cor:jacobian-varying-w} applies on $\mathcal{R}^{w}_{M}$. Taking $\mathbf{r}(v) = v\,\mathbf{p}(v)$ in~\eqref{eq:moment-map-varying-w}, so that
    $\left\langle\mathbf{r}\mathbf{p}^{T}\right\rangle = \mathcal{K}$ by~\eqref{eq:gram-matrices} and the correction sum equals $\mathcal{G}_{v}$, the
    flux~\eqref{eq:moment-system-1d-closure} has the Jacobian
    \begin{equation}
        \mathbf{F}'(\mathbf{m}) = \mathcal{K}\mathcal{H}^{-1}\left(\mathbf{I}+\lambda\mathcal{G}\right) - \lambda\mathcal{G}_{v}
        = \mathcal{K}\mathcal{H}^{-1} + \lambda\left(\mathcal{K}\mathcal{H}^{-1}\mathcal{G}-\mathcal{G}_{v}\right),
    \end{equation}
    which is~\eqref{eq:perturbed-jacobian}. The product $\mathcal{H}^{-1}\mathbf{F}' = \mathcal{H}^{-1}\mathcal{K}\mathcal{H}^{-1} + \lambda\mathcal{H}^{-1}R$ has a
    symmetric leading term but an in general non-symmetric $O(\lambda)$ correction, so the symmetrizer of Theorem~\ref{thm:hyperbolicity-fixed-w} is no longer applicable, since $h_l(\cdot)$ itself now depends on the state through $\kappa_{l}(\mathbf{m})$, so~\eqref{eq:entropy-pair} no longer defines a function of
    $\mathbf{m}$ whose gradient is $-\mathbf{y}^{\ast}$.

    For the spectrum, let $S := \mathcal{H}^{-1/2}\mathcal{K}\mathcal{H}^{-1/2}$, which is symmetric, and note that
    $\mathcal{H}^{-1/2}\left(\mathcal{K}\mathcal{H}^{-1}\right)\mathcal{H}^{1/2} = S$, so $S$ has the eigenvalues $\theta_{1}<\ldots<\theta_{M}$ of
    Theorem~\ref{thm:hyperbolicity-fixed-w}. Writing $\tilde{R} := \mathcal{H}^{-1/2}R\mathcal{H}^{1/2}$, we have that the flux Jacobian is
    similar to $S+\lambda\tilde{R}$. Consider the parameterization $S_{t} := S + t\lambda\tilde{R}$, $t\in[0,1]$. Since $S$ is symmetric, and therefore normal, the Bauer--Fike theorem~\cite{bauer1960norms} applies with condition number one and gives
    \begin{equation}
        \mathrm{spec}\left(S_{t}\right)\subset\bigcup_{i=1}^{M}\overline{D}\left(\theta_{i},t\delta\right)\subseteq\bigcup_{i=1}^{M}\overline{D}\left(\theta_{i},\delta\right)
        \qquad\text{for all } t\in[0,1],
    \end{equation}
    where $\overline{D}(\theta_{i},\delta)$ is a closed disc with center $\theta_{i}$ and radius $\delta := \lambda\|\tilde{R}\|_{2}$.
    Assume~\eqref{eq:bauer-fike-condition}, i.e. $\delta < \tfrac{1}{2}\min_{i\ne j}|\theta_{i}-\theta_{j}|$, and pick
    $\delta' \in \left(\delta,\tfrac{1}{2}\min_{i\ne j}|\theta_{i}-\theta_{j}|\right)$. The open discs $D(\theta_{i},\delta')$ are then pairwise disjoint, their union contains
    $\mathrm{spec}(S_{t})$ for every $t$, and no eigenvalue of any $S_{t}$ lies on their boundaries. The number of eigenvalues of $S_{t}$ inside each such disc, counted
    with multiplicity, is therefore an integer-valued continuous function of $t$, and therefore constant; at $t=0$ it equals one for each $i$, because the $\theta_{i}$ are
    simple. Consequently $S_{1}$, and with it $\mathbf{F}'(\mathbf{m})$ (due to the similarity), has exactly one simple eigenvalue $\tilde{\theta}_{i}$ in each
    $\overline{D}(\theta_{i},\delta)$, so the $M$ eigenvalues are pairwise distinct and $\left|\tilde{\theta}_{i}-\theta_{i}\right|\le\delta$.

    It remains to show that they are real. The matrix $\mathbf{F}'(\mathbf{m})$ is real, so its spectrum is invariant under complex
    conjugation. If some $\tilde{\theta}_{i}$ were not real, its complex conjugate would be a second eigenvalue of
    $\mathbf{F}'(\mathbf{m})$, distinct from $\tilde{\theta}_{i}$, and it would lie in $\overline{D}(\theta_{i},\delta)$ as well,
    that disc being symmetric about the real axis because its centre $\theta_{i}$ is real. This contradicts the fact that
    $\overline{D}(\theta_{i},\delta)$ contains exactly one eigenvalue of $\mathbf{F}'(\mathbf{m})$. Hence
    $\tilde{\theta}_{i}\in\mathbb{R}$ for every $i$. Having $M$ distinct real eigenvalues, $\mathbf{F}'$ is real diagonalizable
    and the system is strictly hyperbolic.

    Finally, for $\lambda = 0$ the terms $\lambda\mathcal{G}$ and $\lambda\mathcal{G}_{v}$ vanish by Corollary~\ref{cor:jacobian-varying-w}, so that
    $\mathbf{F}' = \mathcal{K}\mathcal{H}^{-1}$ and Theorem~\ref{thm:hyperbolicity-fixed-w} applies unchanged.
\end{proof}

Condition~\eqref{eq:bauer-fike-condition} combined with~\eqref{eq:speed-bound} also yields
the wave-speed bound $\max_{i}|\tilde{\theta}_{i}| < \max_{l}|v_{l}| + \lambda\|\mathcal{H}^{-1/2}R\mathcal{H}^{1/2}\|_{2}$. Thus, the consequence of
Theorems~\ref{thm:hyperbolicity-fixed-w} and~\ref{thm:hyperbolicity-varying-w} is that freezing $\mathbf{w}$, or refitting it once per time step and treating it as
frozen within the step, makes the closed system unconditionally and exactly symmetric hyperbolic for every $\lambda$, whereas a fully state-dependent weighting
introduces an $O(\lambda)$ symmetry defect and makes strict hyperbolicity conditional on~\eqref{eq:bauer-fike-condition}.

\begin{remark}
    The restriction to $d=1$ in Theorems~\ref{thm:hyperbolicity-fixed-w} and~\ref{thm:hyperbolicity-varying-w} concerns strictness only. In $d\ge 2$ the argument for
    part~(i) carries over verbatim, the flux Jacobian in the $i$-th spatial direction being $\mathcal{K}^{(i)}\mathcal{H}^{-1}$ with $\mathcal{K}^{(i)}$
    defined in~\eqref{eq:gram-matrices} and symmetric; the closed system is therefore still symmetric hyperbolic. Strictness, however, fails generically: isotropy of
    the constraint set produces repeated eigenvalues already at a Maxwellian state. The multi-dimensional counterpart of~\eqref{eq:speed-bound}, i.e. the speed bound
    $\max_{k}|\theta^{(i)}_{k}| < \max_{l}|\mathbf{v}_{l,i}|$, remains valid: the positive-semidefiniteness argument recalled after
    Theorem~\ref{thm:hyperbolicity-fixed-w} applies verbatim to $\mathcal{K}^{(i)}$. The perturbed Jacobian of Theorem~\ref{thm:hyperbolicity-varying-w} likewise
    carries over: Corollary~\ref{cor:jacobian-varying-w} with $\mathbf{r} = v_{i}\,\mathbf{p}$ gives
    $\mathcal{K}^{(i)}\mathcal{H}^{-1}\left(\mathbf{I}+\lambda\mathcal{G}\right) - \lambda\mathcal{G}^{(i)}_{v}$ in the $i$-th direction, with
    $\mathcal{G}^{(i)}_{v}$ defined as $\mathcal{G}_{v}$ but with $v_{l,i}$ in place of $v_{l}$; the spectral argument, however, relies on the simplicity of the
    unperturbed eigenvalues and therefore does not extend.
\end{remark}

\begin{table}[htb]
\caption{Structure of the closed one-dimensional moment system~\eqref{eq:moment-system-1d}--\eqref{eq:moment-system-1d-closure} as a function of the weighting
$\mathbf{w}$ and the sparsity parameter $\lambda$.}\label{tab:hyperbolicity-summary}
\begin{tabular}{@{}p{0.13\textwidth}p{0.37\textwidth}p{0.37\textwidth}@{}}
\toprule
 & $\lambda = 0$ & $\lambda > 0$ \tabularnewline
\midrule
\raggedright $\mathbf{w}$ frozen &
\raggedright Symmetric hyperbolic with symmetrizer $\mathcal{H}^{-1}$ and entropy pair $(\eta,\psi)$; strictly hyperbolic, with speeds the Gauss nodes of $\mu^{\ast}$ in
$(v_{1},v_{N})$; valid on all of $\mathcal{R}_{M}$ (Theorem~\ref{thm:hyperbolicity-fixed-w}). &
\raggedright Unchanged. By~\eqref{eq:gibbs-form} the penalty only tilts the prior to $\tilde{\phi}_{l} = e^{-1-\lambda\kappa_{l}}$, so
Theorem~\ref{thm:hyperbolicity-fixed-w} applies verbatim; only the values of $\mathbf{f}^{\ast}$, and thus of the $\theta_{i}$, move. \tabularnewline
\addlinespace
\raggedright $\mathbf{w} = \mathbf{w}(\mathbf{m})$ &
\raggedright Unchanged. By~\eqref{eq:gibbs-form}, $\tilde{\phi}_{l} = e^{-1}$ and $\mathbf{f}^{\ast}$ does not depend on $\mathbf{w}$ at all, so
Theorem~\ref{thm:hyperbolicity-fixed-w} applies verbatim; $\mathbf{w}$ acts only as a preconditioner. &
\raggedright $\mathbf{F}' = \mathcal{K}\mathcal{H}^{-1} + \lambda R$. The symmetrizer $\mathcal{H}^{-1}$ and the entropy structure are lost at $O(\lambda)$; strict
hyperbolicity only under~\eqref{eq:bauer-fike-condition}, and only on $\mathcal{R}^{w}_{M}$ (Theorem~\ref{thm:hyperbolicity-varying-w}). \tabularnewline
\botrule
\end{tabular}
\end{table}

Table~\ref{tab:hyperbolicity-summary} summarizes how the two modelling choices, i.e. whether the weighting $\mathbf{w}$ is frozen or refitted from the state, and
whether the sparsity penalty is active, affect the structure of the full moment system in the one-dimensional case.

\subsection{Numerical schemes for transport}
Next, we analyze some numerical schemes for computing the convective part.
For simplicity, we consider one-dimensional transport in the $x$-direction, but keep the distribution function
multi-dimensional. We also assume a uniform spatial grid with a cell size $\Delta x$.

\subsubsection{First-order schemes}
We first note that both the moments and the flux in a cell are obtained from the same reconstructed $\mathbf{f}$ by applying a fixed linear map. Defining the diagonal matrix
of streaming speeds
\begin{equation}
    \mathcal{V}_{x} := \mathrm{diag}\left(v_{1,x},\ldots,v_{N,x}\right)\in\mathbb{R}^{N\times N},
    \label{eq:streaming-matrix}
\end{equation}
and using that the $x$-flux of the $k$-th constrained moment is
$\left\langle v_{x}p_{k}\right\rangle = \sum_{l}\Delta v_{l}v_{l,x}p_{k}(\mathbf{v}_{l})f^{\ast}_{l}$, we have
\begin{equation}
    \mathbf{m}_{i}^{n} = \mathbf{A}\mathbf{f}^{\ast,n}_{i},\qquad
    \mathbf{F}_{i}^{n} = \mathbf{A}\mathcal{V}_{x}\mathbf{f}^{\ast,n}_{i},\label{eq:moment-flux-linear}
\end{equation}
where $\mathbf{f}^{\ast,n}_{i}$ is the reconstruction in cell $i$ at timestep $n$, $\mathbf{m}^{n}_{i}$ the corresponding vector of moments and
$\mathbf{F}^{n}_{i}$ the corresponding flux. The second identity in~\eqref{eq:moment-flux-linear} reduces in the one-dimensional setting
of~\eqref{eq:moment-system-1d-closure} to $\mathbf{A}\mathcal{V}_{x} = \left(\mathbf{A}_{2},\ldots,\mathbf{A}_{M},\mathbf{A}_{\mathrm{next}}\right)^{T}$
via $vp_{k} = p_{k+1}$.

All the first-order schemes considered here are members of a single family, which we analyze below. We use the conservative update
\begin{equation}
    \mathbf{m}^{n+1}_{i} = \mathbf{m}^{n}_{i}-\frac{\Delta t}{\Delta x}\left(\mathbf{F}^{\mathrm{lo}}_{i+1/2}-\mathbf{F}^{\mathrm{lo}}_{i-1/2}\right),
    \label{eq:conservative-update}
\end{equation}
with $\Delta t$ the timestep, and take the node-wise numerical flux
\begin{equation}
    F^{\mathrm{lo}}_{l,i+1/2} = \frac{v_{l,x}}{2}\left(f^{\ast,n}_{i,l}+f^{\ast,n}_{i+1,l}\right)
    - \frac{a_{l,i+1/2}}{2}\left(f^{\ast,n}_{i+1,l}-f^{\ast,n}_{i,l}\right),\label{eq:lo-flux-kinetic}
\end{equation}
whose moment counterpart, obtained by applying $\mathbf{A}$ and using~\eqref{eq:moment-flux-linear}, is
\begin{equation}
    \begin{aligned}
        \mathbf{F}^{\mathrm{lo}}_{i+1/2} &= \frac{1}{2}\left(\mathbf{F}^{n}_{i}+\mathbf{F}^{n}_{i+1}\right)
        - \frac{1}{2}\mathbf{A}\mathcal{D}_{i+1/2}\left(\mathbf{f}^{\ast,n}_{i+1}-\mathbf{f}^{\ast,n}_{i}\right),\\
        \mathcal{D}_{i+1/2} &:= \mathrm{diag}\left(a_{1,i+1/2},\ldots,a_{N,i+1/2}\right).
    \end{aligned}\label{eq:lo-flux-family}
\end{equation}
This is a central flux plus a numerical viscosity $\mathcal{D}_{i+1/2}$ that acts node-wise on the velocity grid through the non-negative coefficients $a_{l,i+1/2}$, both terms being
exact images under $\mathbf{A}$ of quantities assembled from the reconstructions of the two adjacent cells; the viscosities are allowed to differ from interface to
interface. Writing $\nu_{l} := v_{l,x}\Delta t/\Delta x$ for the node-wise Courant number and
$\xi_{l,i+1/2} := a_{l,i+1/2}\Delta t/\Delta x$ for the scaled viscosity, the node-wise update given by~\eqref{eq:lo-flux-kinetic} reads
\begin{equation}
    \tilde{f}_{i,l} = \frac{\xi_{l,i-1/2}+\nu_{l}}{2}f^{\ast,n}_{i-1,l}
    + \left(1-\frac{\xi_{l,i-1/2}+\xi_{l,i+1/2}}{2}\right)f^{\ast,n}_{i,l}
    + \frac{\xi_{l,i+1/2}-\nu_{l}}{2}f^{\ast,n}_{i+1,l},\label{eq:kinetic-lo-componentwise}
\end{equation}
which is the standard three-point scheme for the scalar advection equation $\partial_{t}f_{l}+v_{l,x}\partial_{x}f_{l} = 0$ for the value of the velocity distribution function at the $l$-th velocity node.

The schemes of interest are recovered by the choices
\begin{equation}
    a_{l,i+1/2} =
    \begin{cases}
        \Delta x/\Delta t, & \text{Lax-Friedrichs},\\[2pt]
        \left|v_{l,x}\right|, & \text{upwind},\\[2pt]
        \max_{k}\left|v_{k,x}\right|, & \text{global Rusanov},
    \end{cases}\label{eq:viscosity-choices}
\end{equation}
which we now discuss in turn.

For the Lax-Friedrichs choice $a_{l,i+1/2} = \Delta x/\Delta t$ the viscosity matrix is a multiple of the identity, so that
$\mathbf{A}\mathcal{D}_{i+1/2}\left(\mathbf{f}^{\ast,n}_{i+1}-\mathbf{f}^{\ast,n}_{i}\right) = \frac{\Delta x}{\Delta t}\left(\mathbf{m}^{n}_{i+1}-\mathbf{m}^{n}_{i}\right)$
and~\eqref{eq:lo-flux-family} becomes the usual flux
$\frac{1}{2}\left(\mathbf{F}^{n}_{i}+\mathbf{F}^{n}_{i+1}\right)-\frac{\Delta x}{2\Delta t}\left(\mathbf{m}^{n}_{i+1}-\mathbf{m}^{n}_{i}\right)$, assembled from
the moment arrays alone. The update~\eqref{eq:conservative-update} then can be re-written in the usual Lax-Friedrichs form
\begin{equation}
    \mathbf{m}_{i}^{n+1} = \frac{1}{2} \left(\mathbf{m}_{i+1}^{n} + \mathbf{m}_{i-1}^{n}\right)
    - \frac{\Delta t}{2\Delta x}\left(\mathbf{F}_{i+1}^{n} - \mathbf{F}_{i-1}^{n}\right),\label{eq:lax-friedrichs}
\end{equation}
with the corresponding node-wise update being
\begin{equation}
    \tilde{f}_{i,l} = \frac{1+\nu_{l}}{2}f^{\ast,n}_{i-1,l} + \frac{1-\nu_{l}}{2}f^{\ast,n}_{i+1,l}.\label{eq:kinetic-lax-friedrichs-componentwise}
\end{equation}

For the upwind choice $a_{l,i+1/2} = \left|v_{l,x}\right|$ the viscosity is no longer proportional to the identity and is velocity grid-dependent. Splitting the
streaming speeds into their positive and negative parts, $v^{\pm}_{l,x} := \tfrac{1}{2}\left(v_{l,x}\pm\left|v_{l,x}\right|\right)$, and writing
$\mathcal{V}^{\pm}_{x} := \mathrm{diag}\left(v^{\pm}_{1,x},\ldots,v^{\pm}_{N,x}\right)$, so that $\mathcal{V}_{x} = \mathcal{V}^{+}_{x}+\mathcal{V}^{-}_{x}$ and
$\mathcal{D}_{i+1/2} = \mathcal{V}^{+}_{x}-\mathcal{V}^{-}_{x}$, the flux~\eqref{eq:lo-flux-kinetic} becomes the first-order kinetic upwind flux
\begin{equation}
    F^{\mathrm{up}}_{l,i+1/2} = v^{+}_{l,x}f^{\ast,n}_{i,l} + v^{-}_{l,x}f^{\ast,n}_{i+1,l},\qquad
    \mathbf{F}^{\mathrm{up}}_{i+1/2} = \mathbf{A}\left(\mathcal{V}^{+}_{x}\mathbf{f}^{\ast,n}_{i} + \mathcal{V}^{-}_{x}\mathbf{f}^{\ast,n}_{i+1}\right),
    \label{eq:upwind-flux}
\end{equation}
the second expression following from~\eqref{eq:moment-flux-linear}. Introducing the pair of \emph{upwind moment measurement matrices}
\begin{equation}
    \mathbf{A}^{\pm} := \mathbf{A}\mathcal{V}^{\pm}_{x}\in\mathbb{R}^{M\times N},\qquad
    \mathbf{A}^{\pm}_{kl} = \Delta v_{l}v^{\pm}_{l,x}p_{k}(\mathbf{v}_{l}),\label{eq:upwind-measurement-matrix}
\end{equation}
which measure the same $M$ monomials as $\mathbf{A}$ in~\eqref{eq:A-def}, weighted by the velocity and restricted to the half of the grid moving in the positive
($+$) and negative ($-$) $x$-direction respectively, and which are fixed matrices by assumptions~\ref{assumption:fixed-matrix}
and~\ref{assumption:fixed-grid}, the interface $x_{i+1/2}$ carries not one flux but two, one per sign,
\begin{equation}
    \mathbf{F}^{+}_{i+1/2} = \mathbf{A}^{+}\mathbf{f}^{\ast,n}_{i},\qquad
    \mathbf{F}^{-}_{i+1/2} = \mathbf{A}^{-}\mathbf{f}^{\ast,n}_{i+1},\qquad
    \mathbf{F}^{\mathrm{up}}_{i+1/2} = \mathbf{F}^{+}_{i+1/2}+\mathbf{F}^{-}_{i+1/2},\label{eq:upwind-partial-fluxes}
\end{equation}
each of which is again an exact moment vector of a single reconstruction, measured with $\mathbf{A}^{\pm}$ instead of with $\mathbf{A}\mathcal{V}_{x}$; since
$\mathcal{V}_{x} = \mathcal{V}^{+}_{x}+\mathcal{V}^{-}_{x}$, the two split the flux map of~\eqref{eq:moment-flux-linear},
$\mathbf{A}\mathcal{V}_{x} = \mathbf{A}^{+}+\mathbf{A}^{-}$. The upwinding thus happens in the measurement matrix, not in the moments:
$\mathbf{F}^{\pm}_{i+1/2}$ depends only on the cell lying upwind of the interface for the corresponding sign, and no averaging or characteristic decomposition of
the moment vector is involved. The node-wise update~\eqref{eq:kinetic-lo-componentwise} correspondingly reads
\begin{equation}
    \tilde{f}_{i,l} = \left(1-\left|\nu_{l}\right|\right)f^{\ast,n}_{i,l} + \nu^{+}_{l}f^{\ast,n}_{i-1,l} - \nu^{-}_{l}f^{\ast,n}_{i+1,l},
    \qquad \nu^{\pm}_{l} := v^{\pm}_{l,x}\frac{\Delta t}{\Delta x},\label{eq:kinetic-upwind-componentwise}
\end{equation}
since $\nu^{+}_{l}-\nu^{-}_{l} = |\nu_{l}|$.

Finally, the global Rusanov choice $a_{l,i+1/2} = \max_{k}\left|v_{k,x}\right|$ is the state-independent bound on the characteristic speeds provided by
Theorem~\ref{thm:hyperbolicity-fixed-w}. Like the Lax-Friedrichs viscosity it is proportional to the identity, so that the correction term
in~\eqref{eq:lo-flux-family} is again $\tfrac{1}{2}\max_{k}\left|v_{k,x}\right|\left(\mathbf{m}^{n}_{i+1}-\mathbf{m}^{n}_{i}\right)$, but it is smaller than the
Lax-Friedrichs one whenever the timestep is below the stability limit.

Next, we discuss the treatment of boundary conditions. Here, we propose a first-order \textit{direct kinetic} treatment: since we already have access to the
underlying distribution $f$ obtained via reconstruction, we can use it to directly compute the boundary fluxes. We follow the first-order finite volume
construction of Baranger et al.~\cite{baranger2019numerical}, transposed to the present moment setting, and use the upwind
flux~\eqref{eq:upwind-flux} at the boundary interface, which as shown below makes the numerical mass flux through a wall vanish exactly. Since the viscosities
in~\eqref{eq:lo-flux-family} may be chosen per interface, this is compatible with any choice in the interior.

For a fully diffuse wall we additionally need the ghost state $\mathbf{f}^{\ast,n}_{0}$ in the cell adjacent to the boundary at $x_{1/2}$, where the normal
$\mathbf{n}$ is taken to point into the gas, so that $\mathbf{v}_{l}\cdot\mathbf{n} = v_{l,x}$. Writing
\begin{equation}
    \mathcal{M}^{w}_{l} := \left(2\pi R T_{w}\right)^{-d/2}
    \exp\left(-\frac{\left|\mathbf{v}_{l}-\mathbf{u}_{w}\right|^{2}}{2RT_{w}}\right)\label{eq:wall-maxwellian}
\end{equation}
for the wall Maxwellian evaluated at the $l$-th node (the scaling of the density is arbitrary, as it cancels out in the subsequent formulae), where $T_{w}$ is the wall temperature, $\mathbf{u}_{w}$ is the wall velocity satisfying
$\mathbf{u}_{w}\cdot\mathbf{n} = 0$, and $R$ the specific gas constant, the ghost state is defined node-wise by
\begin{equation}
    f^{\ast,n}_{0,l} =
    \begin{cases}
        f^{\ast,n}_{1,l}, & v_{l,x} < 0,\\[4pt]
        \sigma_{w}\mathcal{M}^{w}_{l}, & v_{l,x} > 0,
    \end{cases}
    \qquad
    \sigma_{w} = -\frac{\sum_{l:v_{l,x}<0}\Delta v_{l}v_{l,x}f^{\ast,n}_{1,l}}
                       {\sum_{l:v_{l,x}>0}\Delta v_{l}v_{l,x}\mathcal{M}^{w}_{l}}.\label{eq:diffuse-wall}
\end{equation}
Outgoing velocities are thus extrapolated to zeroth order from the adjacent cell, whereas incoming velocities are re-emitted according to the diffuse reflection
law, with $\sigma_{w}$ fixed by the requirement that no mass cross the wall. Both sums in~\eqref{eq:diffuse-wall} are assembled from quantities the solver already
forms, and $\sigma_{w}>0$, its numerator and denominator being respectively negative and positive.

The main property of~\eqref{eq:diffuse-wall} is that the numerical mass flux through the wall vanishes exactly. Indeed, by
assumption~\ref{assumption:density-constraint} the mass flux is the density component of~\eqref{eq:upwind-flux} at $x_{1/2}$, and since $v^{+}_{l,x}$ vanishes for
outgoing nodes,
\begin{equation}
    \sum_{l}\Delta v_{l}F^{\mathrm{up}}_{l,1/2}
    = \sigma_{w}\sum_{l:v_{l,x}>0}\Delta v_{l}v_{l,x}\mathcal{M}^{w}_{l}
    + \sum_{l:v_{l,x}<0}\Delta v_{l}v_{l,x}f^{\ast,n}_{1,l} = 0
\end{equation}
by the definition of $\sigma_{w}$, so that mass is conserved globally for an internal flow. Since $\mathcal{M}^{w}_{l}>0$, $\sigma_{w}>0$ and
$\mathbf{f}^{\ast,n}_{1}>0$, the ghost state is strictly positive, which is required in the realizability argument of
Theorem~\ref{thm:first-order-realizability} below.

\begin{theorem}\label{thm:first-order-realizability}
    If assumptions \ref{assumption:fixed-matrix}, \ref{assumption:rank}, \ref{assumption:density-constraint} and \ref{assumption:fixed-grid} hold, the initial
    conditions are realizable, the boundary conditions are treated via the direct kinetic approach, and a positive weighting function $\mathbf{w}$ is used for the
    reconstruction, then the update~\eqref{eq:conservative-update} with the flux~\eqref{eq:lo-flux-family} leads to a realizable set of moments everywhere in the
    domain, provided the numerical viscosities satisfy
    \begin{equation}
        \left|v_{l,x}\right|\le a_{l,i+1/2}\le\frac{\Delta x}{\Delta t}\qquad\text{for every node } l \text{ and every interface } x_{i+1/2}.
        \label{eq:viscosity-window}
    \end{equation}
    The admissible condition~\eqref{eq:viscosity-window} is fulfilled exactly when the CFL condition $\max_{l}|v_{l,x}|\frac{\Delta t}{\Delta x}\le1$ is satisfied,
    and all three choices~\eqref{eq:viscosity-choices} then lie in it, the Lax-Friedrichs and upwind schemes being its two limiting endpoints. In particular, the
    Lax-Friedrichs, upwind and global Rusanov schemes all preserve realizability under that CFL condition.
\end{theorem}
\begin{proof}
    We argue by induction on $n$, the induction hypothesis being that $\mathbf{m}^{n}_{i}\in\mathcal{R}_{M}$ for every cell $i$. The base case ($n=0$) is the assumed
    realizability of the initial data.

    We now assume that $\mathbf{m}^{n}_{i}\in\mathcal{R}_{M}$ for every $i$. By Theorem~\ref{thm:existence} each cell then carries a unique reconstruction
    $\mathbf{f}^{\ast,n}_{i}$, and that reconstruction is \emph{strictly positive} in every velocity node. This is the only place where the positivity of
    $\mathbf{w}$ is used, and it is also the reason the statement is insensitive to $\lambda$: by~\eqref{eq:gibbs-form} the penalty and the weighting only tilt the
    prior $\tilde{\phi}_{l} > 0$, so $\mathbf{f}^{\ast,n}_{i} > 0$ for every $\lambda \ge 0$ and every positive $\mathbf{w}$.

    Next, we insert~\eqref{eq:moment-flux-linear} and~\eqref{eq:lo-flux-family} into the update~\eqref{eq:conservative-update}. Since $\mathbf{A}$, $\mathcal{V}_{x}$ and the
    diagonal viscosity matrices $\mathcal{D}_{i\pm1/2}$ are fixed matrices, by assumptions~\ref{assumption:fixed-matrix} and~\ref{assumption:fixed-grid}, and the
    update is affine in $\left(\mathbf{m},\mathbf{F}\right)$, the matrix $\mathbf{A}$ factors out of the entire right-hand side:
    \begin{equation}
        \begin{aligned}
            \mathbf{m}^{n+1}_{i} &= \mathbf{A}\tilde{\mathbf{f}}_{i},\\
            \tilde{\mathbf{f}}_{i} &:= \frac{\Delta t}{2\Delta x}\left(\mathcal{D}_{i-1/2}+\mathcal{V}_{x}\right)\mathbf{f}^{\ast,n}_{i-1}
            + \left(\mathbf{I}-\frac{\Delta t}{2\Delta x}\left(\mathcal{D}_{i-1/2}+\mathcal{D}_{i+1/2}\right)\right)\mathbf{f}^{\ast,n}_{i}\\
            &\qquad + \frac{\Delta t}{2\Delta x}\left(\mathcal{D}_{i+1/2}-\mathcal{V}_{x}\right)\mathbf{f}^{\ast,n}_{i+1}.
        \end{aligned}
        \label{eq:kinetic-lo}
    \end{equation}
    The corresponding update of the $l$-th component of the velocity distribution function is given by~\eqref{eq:kinetic-lo-componentwise}. The moment update is thus the exact image under $\mathbf{A}$ of a node-wise kinetic update; no
    approximation is involved in passing between the two levels, because the moments and the flux are the same fixed linear map applied to the same
    reconstruction.

    Under~\eqref{eq:viscosity-window} the three coefficients in~\eqref{eq:kinetic-lo-componentwise} are non-negative: the outer two because
    $\xi_{l,i\mp1/2}\ge\left|\nu_{l}\right|$, and the middle one because $\xi_{l,i\pm1/2}\le1$. They also sum to one, so $\tilde{f}_{i,l}$ is a convex
    combination of the strictly positive numbers $f^{\ast,n}_{i-1,l}$, $f^{\ast,n}_{i,l}$ and $f^{\ast,n}_{i+1,l}$. Therefore
    $\tilde{\mathbf{f}}_{i}\in\mathbb{R}^{N}_{>0}$, and
    \begin{equation}
        \mathbf{m}^{n+1}_{i} = \mathbf{A}\tilde{\mathbf{f}}_{i} \in \mathbf{A}\left(\mathbb{R}^{N}_{>0}\right) = \mathcal{R}_{M}
    \end{equation}
    by the definition of the realizability cone.


    In cells adjacent to a boundary the stencil in~\eqref{eq:kinetic-lo} uses a ghost state supplied by the direct kinetic treatment. That treatment returns a
    distribution on the same velocity grid, assembled from the reconstructed $\mathbf{f}^{\ast,n}$ of the adjacent cell and, for a diffuse wall, from a wall
    Maxwellian with a non-negative accommodation coefficient; it is therefore itself non-negative, and~\eqref{eq:kinetic-lo-componentwise} again allows writing
    $\mathbf{m}^{n+1}_{i}$ as $\mathbf{A}$ applied to a non-negative vector, i.e. a positive-valued underlying discretized distribution.

    Finally, the interval $\left[\left|v_{l,x}\right|,\Delta x/\Delta t\right]$ is non-empty for every $l$ if and only if
    $\max_{l}\left|v_{l,x}\right|\le\Delta x/\Delta t$, which is the CFL condition; under it the three choices~\eqref{eq:viscosity-choices} evidently
    satisfy~\eqref{eq:viscosity-window}.
\end{proof}

It is worth noting that the proof uses nothing about the family~\eqref{eq:lo-flux-family} beyond the fact that its kinetic
counterpart~\eqref{eq:kinetic-lo-componentwise} is a convex combination of neighbouring states with node-wise non-negative coefficients. Any scheme sharing that
property preserves realizability under the corresponding CFL condition by the same argument, as does any convex combination of such updates, and hence any SSP
Runge--Kutta method built from them.
For an alternative implicit timestepping method that operates directly on the dual variables to avoid inner-loop Newton iterations to determine the moments, the reader is referred to~\cite{schaerer2017efficient}.

\subsubsection{A flux-limited Lax--Wendroff scheme}\label{subsec:lax-wendroff}
The first-order schemes of the previous subsection are robust but strongly diffusive, and we now describe a second-order variant obtained by blending any member
of the family~\eqref{eq:lo-flux-family} with a Lax--Wendroff flux. We keep the conservative form
$\mathbf{m}^{n+1}_{i} = \mathbf{m}^{n}_{i} - \frac{\Delta t}{\Delta x}\left(\mathbf{F}_{i+1/2}-\mathbf{F}_{i-1/2}\right)$ and replace the first-order flux
$\mathbf{F}^{\mathrm{lo}}$ of~\eqref{eq:lo-flux-family}, with viscosities obeying~\eqref{eq:viscosity-window}, by the limited flux
\begin{equation}
    \mathbf{F}_{i+1/2} = \mathbf{F}^{\mathrm{lo}}_{i+1/2}
    + \varphi_{i+1/2}\left(\mathbf{F}^{\mathrm{LW}}_{i+1/2}-\mathbf{F}^{\mathrm{lo}}_{i+1/2}\right).\label{eq:limited-flux}
\end{equation}
Here $\mathbf{F}^{\mathrm{LW}}$ is the higher-order Lax-Wendroff flux, and $\varphi_{i+1/2}\in[0,1]$ is the limiter. The usual approach to computing $\mathbf{F}^{\mathrm{LW}}$ is via the flux Jacobian, which by Theorem~\ref{thm:hyperbolicity-fixed-w} is
$\mathbf{F}' = \mathcal{K}\mathcal{H}^{-1}$; this is in fact can be used here at minimal computational cost, since $\mathcal{H}$ is already formed and factorized by the dual Newton iteration
and $\mathcal{K}$ costs one further accumulation over the velocity grid; this is the approach proposed in~\cite{schaerer2017efficient}. There is, however, an alternative route that requires no linear solve at all. As
in~\eqref{eq:lo-flux-kinetic}, we construct the scheme at the kinetic level, where transport reduces to scalar advection at the constant
speed $v_{l,x}$ carried by the $l$-th node, and then project. The node-wise Lax--Wendroff flux for
$\partial_{t}f_{l}+v_{l,x}\partial_{x}f_{l}=0$ is
\begin{equation}
    F^{\mathrm{LW}}_{l,i+1/2} = \frac{v_{l,x}}{2}\left(f_{l,i}+f_{l,i+1}\right)
    - \frac{v_{l,x}^{2}\Delta t}{2\Delta x}\left(f_{l,i+1}-f_{l,i}\right).\label{eq:kinetic-lax-wendroff}
\end{equation}
Applying $\mathbf{A}$ to~\eqref{eq:kinetic-lax-wendroff} and introducing, in the notation of~\eqref{eq:moment-flux-linear}, the two shifted moment vectors
\begin{equation}
    \mathbf{u}^{(1)}_{i} := \mathbf{A}\mathcal{V}_{x}\mathbf{f}^{\ast}_{i} = \mathbf{F}_{i},\qquad
    \mathbf{u}^{(2)}_{i} := \mathbf{A}\mathcal{V}_{x}^{2}\mathbf{f}^{\ast}_{i},\label{eq:shifted-moments}
\end{equation}
whose $k$-th components are the moments with exponent vectors $\boldsymbol{\alpha}_{k}+\boldsymbol{\epsilon}_{1}$ and
$\boldsymbol{\alpha}_{k}+2\boldsymbol{\epsilon}_{1}$ respectively, gives
\begin{equation}
    \mathbf{F}^{\mathrm{LW}}_{i+1/2} = \frac{1}{2}\left(\mathbf{u}^{(1)}_{i}+\mathbf{u}^{(1)}_{i+1}\right)
    - \frac{\Delta t}{2\Delta x}\left(\mathbf{u}^{(2)}_{i+1}-\mathbf{u}^{(2)}_{i}\right).\label{eq:lax-wendroff-moments}
\end{equation}
The mechanism is that the flux Jacobian of the kinetic system is the diagonal matrix $\mathcal{V}_{x}$ of~\eqref{eq:streaming-matrix}, so its square is again
diagonal, and under $\mathbf{A}$ the squared Jacobian becomes a shift of the monomial exponent by two in the $x$-direction. The quantity $\mathbf{u}^{(1)}$ is the
flux array already required by~\eqref{eq:lo-flux-family}, and $\mathbf{u}^{(2)}$ is one additional higher-order moment computation using the reconstruction
$\mathbf{f}^{\ast}_{i}$; it involves moments one order beyond the closing moments $\mathbf{m}_{\mathrm{next}}$, but these are obtained in exactly the same way via
summation over the velocity grid. As with~\eqref{eq:kinetic-lo}, the moment scheme is the exact image
under $\mathbf{A}$ of a node-wise kinetic scheme, so that the moment and discrete-velocity solvers remain directly comparable.

Since $\mathbf{u}^{(1)}_{i} = \mathbf{F}_{i}$, the leading terms of~\eqref{eq:lo-flux-family} and~\eqref{eq:lax-wendroff-moments} coincide for every choice of
the viscosities, and the antidiffusive flux reduces to the difference of the two numerical viscosities,
\begin{equation}
    \mathbf{F}^{\mathrm{LW}}_{i+1/2}-\mathbf{F}^{\mathrm{lo}}_{i+1/2}
    = \frac{1}{2}\mathbf{A}\mathcal{D}_{i+1/2}\left(\mathbf{f}^{\ast}_{i+1}-\mathbf{f}^{\ast}_{i}\right)
    - \frac{\Delta t}{2\Delta x}\left(\mathbf{u}^{(2)}_{i+1}-\mathbf{u}^{(2)}_{i}\right),\label{eq:antidiffusive-flux}
\end{equation}
which is the standard flux-corrected-transport structure. For the Lax-Friedrichs viscosity the first term is $\frac{\Delta x}{2\Delta t}\left(\mathbf{m}_{i+1}-\mathbf{m}_{i}\right)$
and~\eqref{eq:antidiffusive-flux} is assembled from the moment arrays alone; for the upwind viscosity it is
$\frac{1}{2}\left(\mathbf{A}^{+}-\mathbf{A}^{-}\right)\left(\mathbf{f}^{\ast}_{i+1}-\mathbf{f}^{\ast}_{i}\right)$, and the whole correction collapses node-wise to
\begin{equation}
    F^{\mathrm{LW}}_{l,i+1/2}-F^{\mathrm{up}}_{l,i+1/2}
    = \frac{\left|v_{l,x}\right|}{2}\left(1-\left|\nu_{l}\right|\right)\left(f^{\ast}_{l,i+1}-f^{\ast}_{l,i}\right),\label{eq:antidiffusive-upwind}
\end{equation}
the classical high-resolution correction, which vanishes at $\left|\nu_{l}\right| = 1$; the upwind base scheme is thus the less diffusive starting point, at the
price of a viscosity that is applied on the velocity grid rather than to the moment vector.

The realizability argument of Theorem~\ref{thm:first-order-realizability} does not extend to $\varphi_{i+1/2}>0$: the kinetic Lax--Wendroff update carries the
node-wise coefficients $\tfrac{1}{2}\nu_{l}(\nu_{l}+1)$, $1-\nu_{l}^{2}$ and $\tfrac{1}{2}\nu_{l}(\nu_{l}-1)$, the last of which is negative for
$0<\nu_{l}<1$, so the update is no longer a convex combination of neighbouring states and positivity of the underlying distribution is lost.

We therefore apply flux limiting by using a single scalar $\varphi_{i+1/2}$, shared by all moment components, computed from a small number of
physically meaningful scalar indicators, namely the density $m_{0,\ldots,0}$ and the second-order moments in each velocity direction; for each indicator a smoothness ratio of consecutive
increments is formed, a standard limiter function is evaluated, and $\varphi_{i+1/2}$ is taken as the minimum over the indicators. In the present work,
the van Leer limiter is used. Limiting each moment component
independently would produce a correction vector of no controlled direction relative to $\mathcal{R}_{M}$, whereas a shared scalar confines the updated state to
the segment
\begin{equation}
    \mathbf{m}^{\mathrm{lo}}_{i} + \varphi\left(\mathbf{m}^{\mathrm{LW}}_{i}-\mathbf{m}^{\mathrm{lo}}_{i}\right),\qquad \varphi\in[0,1],
\end{equation}
joining the first-order update, which lies in $\mathcal{R}_{M}$ by Theorem~\ref{thm:first-order-realizability}, to the unlimited second-order update. As
$\mathcal{R}_{M}$ is an open convex cone containing the first endpoint, the set of admissible $\varphi$ is an interval containing $0$, so that $\varphi_{i+1/2}$
may always be reduced until the updated moments are realizable. This provides a one-dimensional backtracking safeguard, with the first-order scheme recovered in
the limit $\varphi_{i+1/2}\to0$.
An alternative approach, described in~\cite{alldredge2012high}, applies limiting on the kinetic, i.e. discrete velocity, level, and therefore requires knowledge of the values $f_l$ of the velocity distribution function in several cells in order to compute the slopes, whereas the limiter formulation presented above operates only on the moment values, but does not provide any a priori feasibility guarantee.
We also refer the reader to~\cite{johnson2023positivity} for a further discussion of limiting approaches in the context of moment methods.

 \section{Overview of numerical algorithm for one-dimensional flows}\label{sec:algorithm}
We provide an outline of the full numerical algorithm for simulation of one-dimensional flows.
We model collisional effects using the BGK collision model~\cite{bhatnagar1954model} with a discretized Maxwellian that \emph{exactly} reproduces the given density, momentum, and energy, found via a Newton iteration (this being the same Maxwellian used for the weighting $\mathbf{w}$), i.e.
the same target distribution proposed in~\cite{mieussens2000discrete}, formulated there as a constrained entropy minimizer. Operator splitting is used to separate the collision and convection parts, and the BGK collisional term
is integrated exactly, assuming a constant collision frequency $\nu$ and target distribution over the time step $\Delta t$, i.e. for the change in the moments $\mathbf{m}$ due to collision is written as
\begin{equation}
    \delta \mathbf{m} = e^{-\nu \Delta t}\mathbf{m} + (1-e^{-\nu \Delta t})\mathbf{m}^{\mathrm{eq}},\label{eq:bgk-update}
\end{equation}
where $\mathbf{m}^{\mathrm{eq}}$ is given by $\mathbf{A}\mathbf{w}$, where $\mathbf{w}$ is the discretized Maxwellian as given by~\eqref{eq:maxwellian-weight}. The collision frequency $\nu$ is computed on the basis of the variable hard sphere model~\cite{DSMC_Bird}. Assuming scaled units, the expression reads
\begin{equation}
    \nu(n, T) = \frac{n}{\mu_{\mathrm{ref}} T^{(\omega_{\mathrm{VHS}} - 1)}},
\end{equation}
where $n$ is the number density, $\omega_{\mathrm{VHS}}$ is the VHS exponent, and $\mu_{\mathrm{ref}}$ is a reference viscosity at a reference temperature $T_{\mathrm{ref}}=1$. 

\begin{algorithm}[h]
\caption{Computation of one-dimensional flows with the sparse entropic quadrature method}
\label{alg:sparse_entropic_flow}

 \hspace*{\algorithmicindent} \textbf{Input:} $f$, maximum order $M_{\max}$ of moments to preserve, regularization $\lambda$,  velocity nodes $\mathbf{v}_{l}$, quadrature weights $\Delta v_{l}$, numerical viscosities $a_l$ (see~\eqref{eq:viscosity-choices}),
 left and right wall temperatures $T_{w,L}$, $T_{w,R}$ and velocities $\mathbf{u}_{w,L}$, $\mathbf{u}_{w,R}$,
 initial number densities, velocities, and temperatures $n^{0}_i$, $\mathbf{u}^{0}_i$, $T^{0}_i$ ($i$ is the index of the grid cell), reference viscosity $\mu_{\mathrm{ref}}$, viscosity exponent $\omega_{\mathrm{VHS}}$,  domain length $L$, number of grid cells $N_x$, timestep $\Delta t$,
 number of timesteps $n_t$  \\
 \hspace*{\algorithmicindent} \textbf{Output:} Macroscopic flow parameters after $n_t$ timesteps
\begin{algorithmic}[1]


\Statex \textbf{Phase 1: Initialization}
\State Compute wall distribution Maxwellians based on $T_{w,L}$, $T_{w,R}$, $\mathbf{u}_{w,L}$, $\mathbf{u}_{w,R}$ via \eqref{eq:wall-maxwellian}
\State Compute moment measurement matrix $\mathbf{A}$ via \eqref{eq:A-def} and $\mathcal{V}_{x}$ via \eqref{eq:streaming-matrix}
\For{$i=1$ to $n_x$}
\State Compute initial distribution in cell $i$ as Maxwellian based on $n^{0}_i$, $\mathbf{u}^{0}_i$, $T^{0}_i$
\State Compute moments $\mathbf{m}_i^{0}$ in cell $i$ using \eqref{eq:moment-compute-linear-def}
\EndFor

\State \textbf{Phase 2: Solution}
\For{$n=1$ to $n_{t}$}
    \For{$i=1$ to $n_x$}
        \State Evaluate local equilibrium distribution $\mathbf{w}$ in cell $i$ according to \eqref{eq:maxwellian-weight} so that density, momentum, and energy constraints are satisfied exactly, as given by \eqref{eq:maxwellian-exact-constraints}
        \State Evaluate equilibrium of $\mathbf{w}$ and apply BGK update to $\mathbf{m}^{n}_i$ according to \eqref{eq:bgk-update}
        \State Reconstruct $\mathbf{f}^{\ast}$ in cell $i$ by solving \eqref{eq:optimization-spec}
        \If{$i==1$ or $i==n_x$}
            \State Compute ghost states according to~\eqref{eq:diffuse-wall}
            \State Compute flux $\mathbf{F}^n_{0}$ and $\mathbf{F}^n_{n_x+1}$ according to according to \eqref{eq:moment-flux-linear} using the ghost state distribution
        \EndIf
        \State Compute flux $\mathbf{F}^n_{i}$ according to \eqref{eq:moment-flux-linear} using the reconstructed distribution $\mathbf{f}^{\ast}$
    \EndFor
    \For{$i=1$ to $n_x$}
        \State Perform Lax-Friedrichs timestep update $\mathbf{m}_{i}^{n+1} = \frac{1}{2} \left(\mathbf{m}_{i+1}^{n} + \mathbf{m}_{i-1}^{n}\right)
    - \frac{\Delta t}{2\Delta x}\left(\mathbf{F}_{i+1}^{n} - \mathbf{F}_{i-1}^{n}\right)$ 
    \EndFor
\EndFor
\end{algorithmic}
\end{algorithm}

An outline of the full set of steps to simulate a flow using the sparse entropic method is given in Algorithm~\ref{alg:sparse_entropic_flow}
for the first-order Lax-Friedrichs scheme. The upwinding scheme is similar, and only requires the 
additional pre-computation of the upwind moment measurement matrices $\mathbf{A}^{\pm}$ according to~\eqref{eq:upwind-measurement-matrix}
and adjustment in as to how the fluxes are evaluated, i.e. using instead~\eqref{eq:upwind-partial-fluxes} of~\eqref{eq:lax-friedrichs}.

\section{Numerical results}\label{sec:numerical results}
In this section, we present numerical results for the two test cases: a Sod shock tube and a Couette flow.
We use a discrete velocity code with an upwinding flux as a baseline numerical scheme against which we compare the moment method approach.
In the present work, a hard sphere model was used, i.e. $\omega_{\mathrm{VHS}} = 0.5$.

In all the results presented below, we consider a full three-dimensional distribution function on a uniform velocity grid and conservation of all moments up to total order $M_{\mathrm{max}}$, i.e. all moments given by the exponent vectors $\boldsymbol{\alpha}_{k} = (\alpha_{k,x},\alpha_{k,y},\alpha_{k,z})$ such that $\alpha_{k,x} + \alpha_{k,y} + \alpha_{k,z} \leq M_{\mathrm{max}}$. The question of conservation of different orders of moments depending on the dominant flow direction is thus left for future work. We take $M_{\mathrm{max}} \in \{2,4,6\}$, corresponding to conservation of 10, 35, and 84 moments, respectively.
We also scale $\lambda$ by $\Delta v$, where $\Delta v$ is the uniform velocity grid spacing; this is done to ensure that the sparsity-promoting regularization strength is independent of the velocity grid spacing. All the results presented below provide the unscaled value of $\lambda$.

The simulations setups used to produce the results are publicly available on GitHub~\cite{oblapenko2026specqkrepro} and are based on version 0.2.0 of the publicly available SPEcQK.jl code~\cite{oblapenko2026specqk} developed by the first author.
The solution data has also been made available on Zenodo~\cite{oblapenko2026specqkdata}.

\subsection{Sod shock tube}
First, we consider the Sod shock tube problem. We solve the equations on a one-dimensional domain of length 1, with the left state given having number density and temperature $(n_L, T_L)=(1,0,1)$ and the right state having number density and temperature $(n_R, T_R)=(\frac{1}{8}, \frac{4}{5})$, corresponding to a 10:1 pressure ratio and 8:1 density ratio. Both states have their velocity set to 0.
We use a velocity grid with extent $[-5,5]$ and 22 nodes in each velocity direction.
We consider the following cases: 1) a case wih $\mu_{\mathrm{ref}} = 2 \times 10^{-2}$, corresponding to a Knudsen number of the left state equal to approximately $2.5 \times 10^{-2}$; 2) a collisionless case without the BGK term.
We solve the equations on a 400-cell grid using a timestep of $4 \times 10^{-4}$, corresponding to a CFL number of 0.8 as governed by the largest grid velocity.

\begin{figure}[t!]
  \centering
  \includegraphics[width=0.88\textwidth]{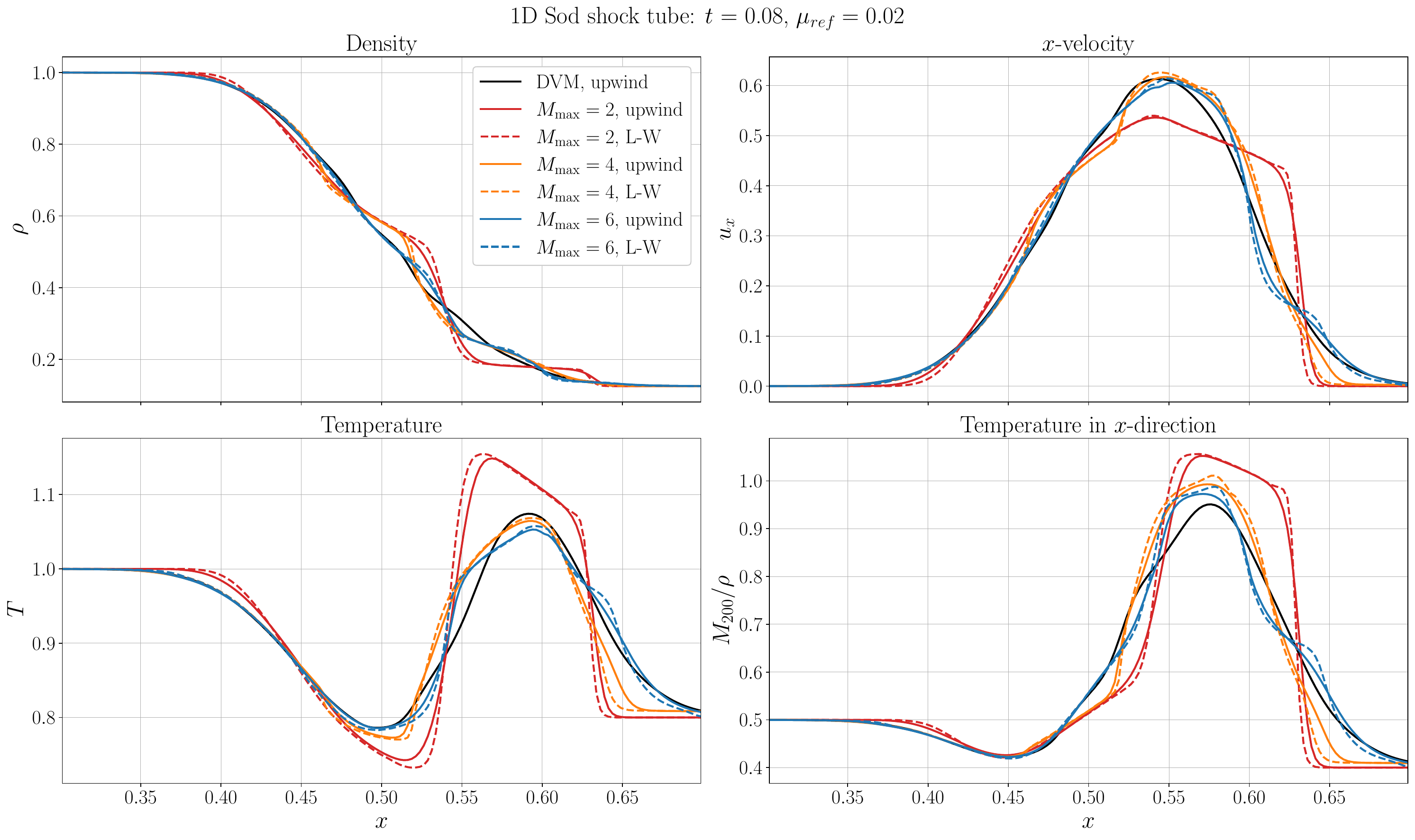}
  \caption{Sod problem profiles for density (upper left), $y$-velocity (upper right), temperature (lower left), and directional temperature (lower right) computed with the DVM method and the sparse entropic quadrature method with $M_{\mathrm{max}}=2$, $M_{\mathrm{max}}=4$, $M_{\mathrm{max}}=6$  and $\lambda=0$; collisional case.}\label{fig:sod002}
  \end{figure}

  Figure~\ref{fig:sod002} shows the solution of the Sod shock tube problem at $t=0.08$ for the case of $\mu_{\mathrm{ref}} = 2 \times 10^{-2}$ and no sparsity-promoting regularization. The DVM solution is computed using the first-order upwinding scheme.
  The moment-based solutions are computed using the first-order upwinding scheme (solid lines) and the flux-limited second-order Lax--Wendroff scheme (dashed lines).
  We observe that the 10-moment solution($M_{\mathrm{max}=2}$) deviates strongly from the reference DVM solution, especially in the temperature profile. Since the 10-moment solution does not explicitly model the heat flux, which is computed only on the basis of the maximum entropy closure, this leads to a significant decrease in the heat transport, thus causing the higher temperatures. The higher-order moment-based solutions, i.e. the the 35-moment solutions ($M_{\mathrm{max}}=4$) and 84-moment solutions ($M_{\mathrm{max}}=6$) exhibit significantly better agreement with the reference DVM solution, with the 84-moment solution showing much better agreement in the low-density region at the right of the plotted domain. Regarding the different numerical advection schemes, we observe that the use of the Lax--Wendroff scheme leads to slightly sharper profiles, however, due to the relatively high grid resolution, even the first-order upwinding scheme already
  produces quite sharp profiles.

\begin{figure}[t!]
  \centering
  \includegraphics[width=0.88\textwidth]{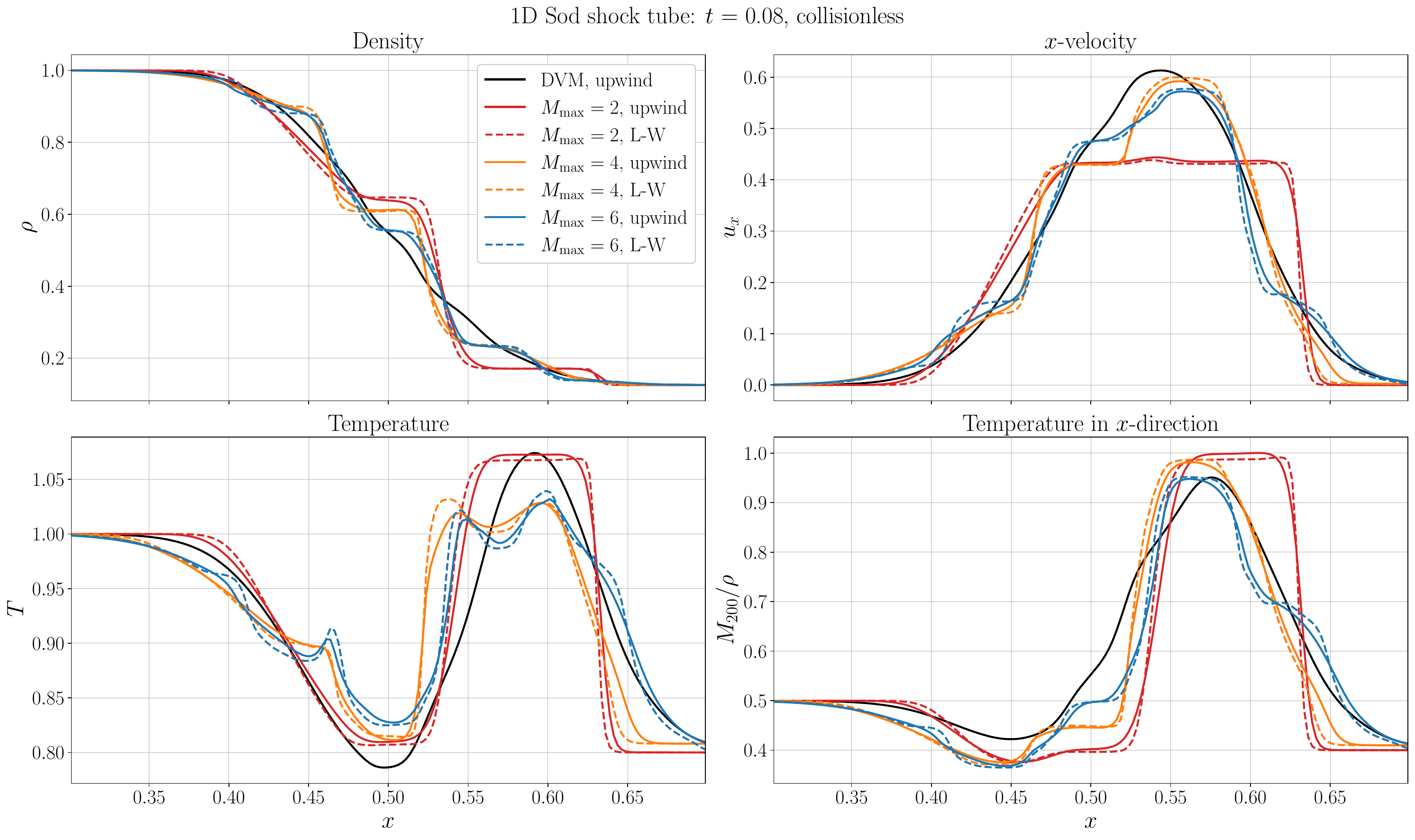}
  \caption{Sod problem profiles for density (upper left), $y$-velocity (upper right), temperature (lower left), and normal stress (lower right) computed with the DVM method and the sparse entropic quadrature method with $M_{\mathrm{max}}=2$, $M_{\mathrm{max}}=4$ and $\lambda=0$; collisionless case.}\label{fig:sodnocoll}
  \end{figure}

Figure~\ref{fig:sodnocoll} shows the profiles for the collisionless case. All the moment solutions deviate stronger from the reference solution than in the collisional case, as the reference solution is strongly discontinuous due to lack of any collisional term that would force it to relax towards a Maxwellian, and cannot be accurately represented by the exponential terms resulting from the maximum-entropy closure~\cite{boccelli2024gallery}. However, increasing the number of moments still allows for more accurate representation of the density, velocity, and temperature profiles, especially towards the edges of the domain. Again, the lack of an explicitly modelled heat flux term in the moment equations for $M_{\mathrm{max}}=2$ leads to a poor representation of the temperature profile, and higher-order moment equations produce results that are significantly closer between themselves. As in the collisional case, the Lax--Wendroff scheme produces slightly sharper profiles of all the macroscopic quantities considered.

\begin{figure}[t!]
  \centering
  \includegraphics[width=0.88\textwidth]{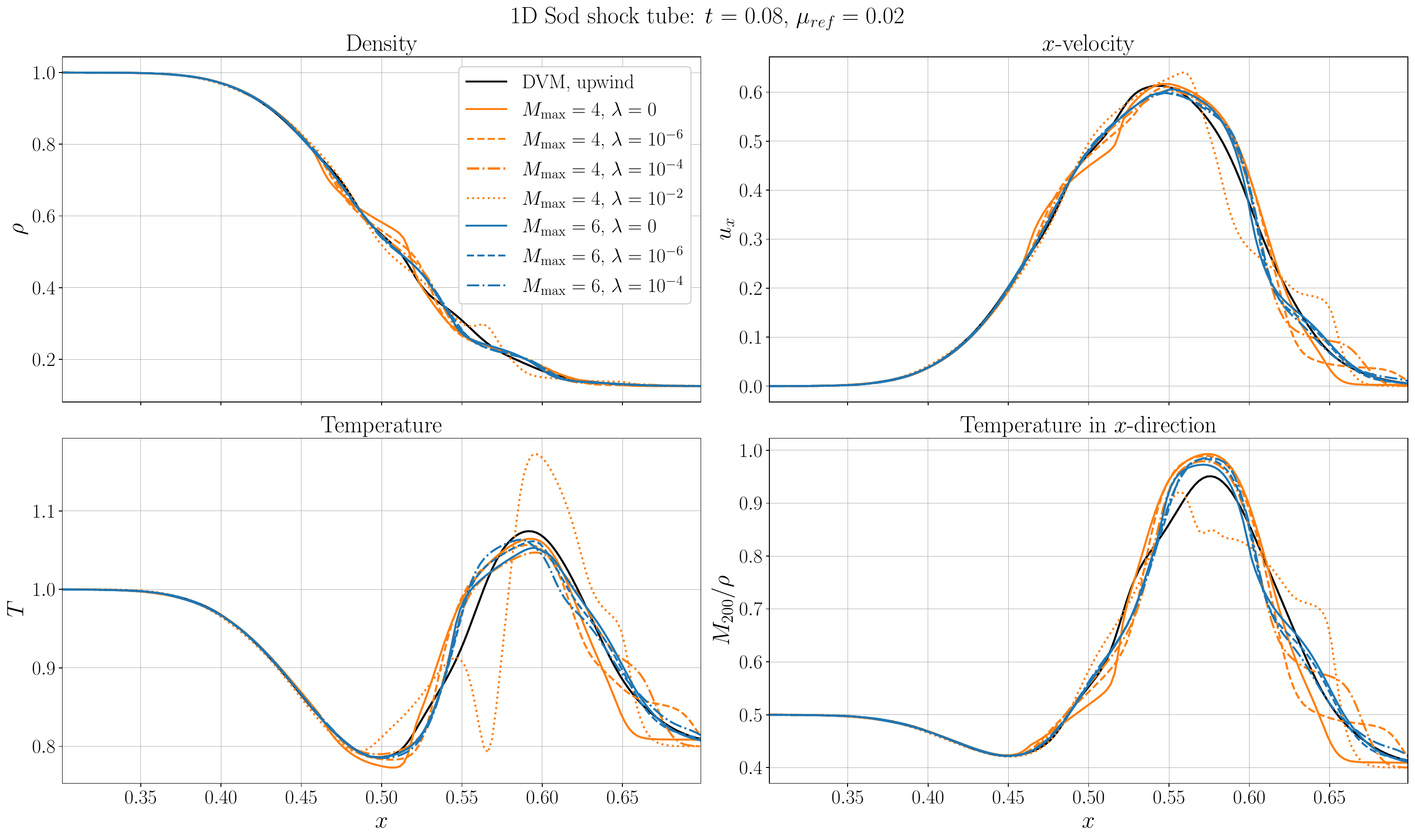}
  \caption{Sod problem profiles for density (upper left), $y$-velocity (upper right), temperature (lower left), and directional temperature (lower right) computed with the DVM method and the sparse entropic quadrature method with $M_{\mathrm{max}}=4$ and $M_{\mathrm{max}}=6$ for different values of $\lambda$; collisional case.}\label{fig:sod002-sparsity}
  \end{figure}

Next, we study the impact of the sparsity-promoting regularization on the solution by varying the $\lambda$ parameter. We omit the $M_{\mathrm{max}}=2$ case due to its general poor quality of approximation, and consider only the upwinding scheme due to the relatively small differences observed between it and the Lax--Wendroff method. Figure~\ref{fig:sod002-sparsity} shows the computed solutions for various values of $\lambda$. It can immediately be observed that for the highest value of $\lambda=10^{-2}$, the $M_{\mathrm{max}}=4$ solution deviates very strongly from both the reference and the other moment-based solutions, and exhibits unphysical minima and maxima. For the same value of $\lambda$ and $M_{\mathrm{max}}=6$, the solver did not converge, and thus no solution is shown. However, for smaller values of $\lambda$, the solutions are in good agreement with the reference solution, especially for the $M_{\mathrm{max}}=6$ case, whereas the $M_{\mathrm{max}}=4$ case exhibits some discrepancies in the right part of the domain, especially in the temperature profiles.

\begin{figure}[t!]
  \centering
  \includegraphics[width=0.88\textwidth]{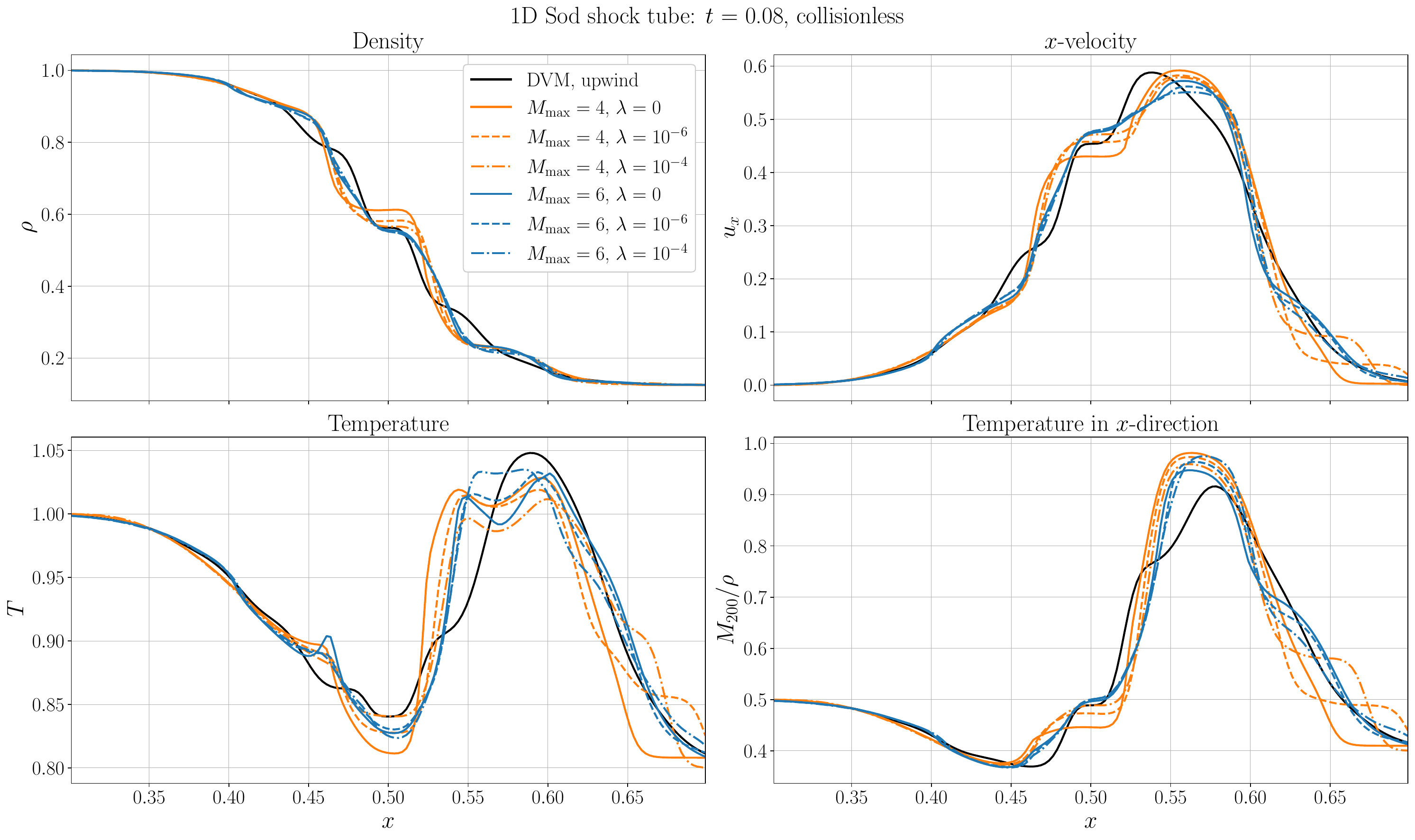}
  \caption{Sod problem profiles for density (upper left), $y$-velocity (upper right), temperature (lower left), and directional temperature (lower right) computed with the DVM method and the sparse entropic quadrature method with $M_{\mathrm{max}}=4$ and $M_{\mathrm{max}}=6$ for different values of $\lambda$; collisionless case.}\label{fig:sodnocoll-sparsity}
  \end{figure}

Similar behaviour is observed for the converged solutions in the collisionless case, as shown on Figure~\ref{fig:sodnocoll-sparsity}. However,
for the highest value of the regularization parameter ($\lambda=10^{-2}$), the solver fails to converge both for $M_{\mathrm{max}}=4$ and $M_{\mathrm{max}}=6$. For lower values of the regularization strength, the solutions remain quite close to the non-sparse ones, although the deviations are larger than in the collisional case.

Finally, we discuss how the different values of $\lambda$ affect the sparsity of the computed underlying distribution. As the latter is non-negative by construction (see Eqn.~\ref{eq:gibbs-form}), sparsity is not exact. Therefore, we use a threshold of $10^{-10}$ for the values of the function $g$ to determine whether a value is zero, which corresponds to the reconstructed distribution have a value 10 orders of magnitude smaller than that of the local Maxwell--Boltzmann distribution.
For $\lambda=0$, all moment method-based solutions have a sparsity of 0\%; for $\lambda=10^{-6}$, the sparsity is approximately 83\% across all moment solutions, for $\lambda=10^{-4}$ it is roughly 90\%, and for $\lambda=10^{-2}$ it is 95\%. We therefore conclude that even with a high degree of sparsity (up to 90\%), the solutions produced by the moment method with the sparse entropic closure are still accurate and retain the main features of the non-sparse solutions; for higher degrees of sparsity, the solution quality degrades and the solver becomes unstable. Finally, we note that for a fixed value of $\lambda$, the resulting sparsity is virtually unaffected by the number of moment constraints used.

\subsection{Couette flow}
Next, we consider a Couette flow between in a channel of length 1 between two parallel plates with a wall temperature $T_w = 1$ and moving in the $y$-direction with a velocity of $\mp 1$ (left and right wall, respectively). We take the reference viscosity as $3 \times 10^{-1}$, corresponding to a Knudsen number of approximately $4\times 10^{-1}$, placing the flow firmly in the transitional regime.
We use a velocity grid with extent $[-6,6]$ and 16 nodes in each velocity direction.

First, we briefly analyze the reference discrete velocity results before comparing them with the results computed with the sparse entropic quadrature method.

\begin{figure}[t!]
  \centering
  \includegraphics[width=0.88\textwidth]{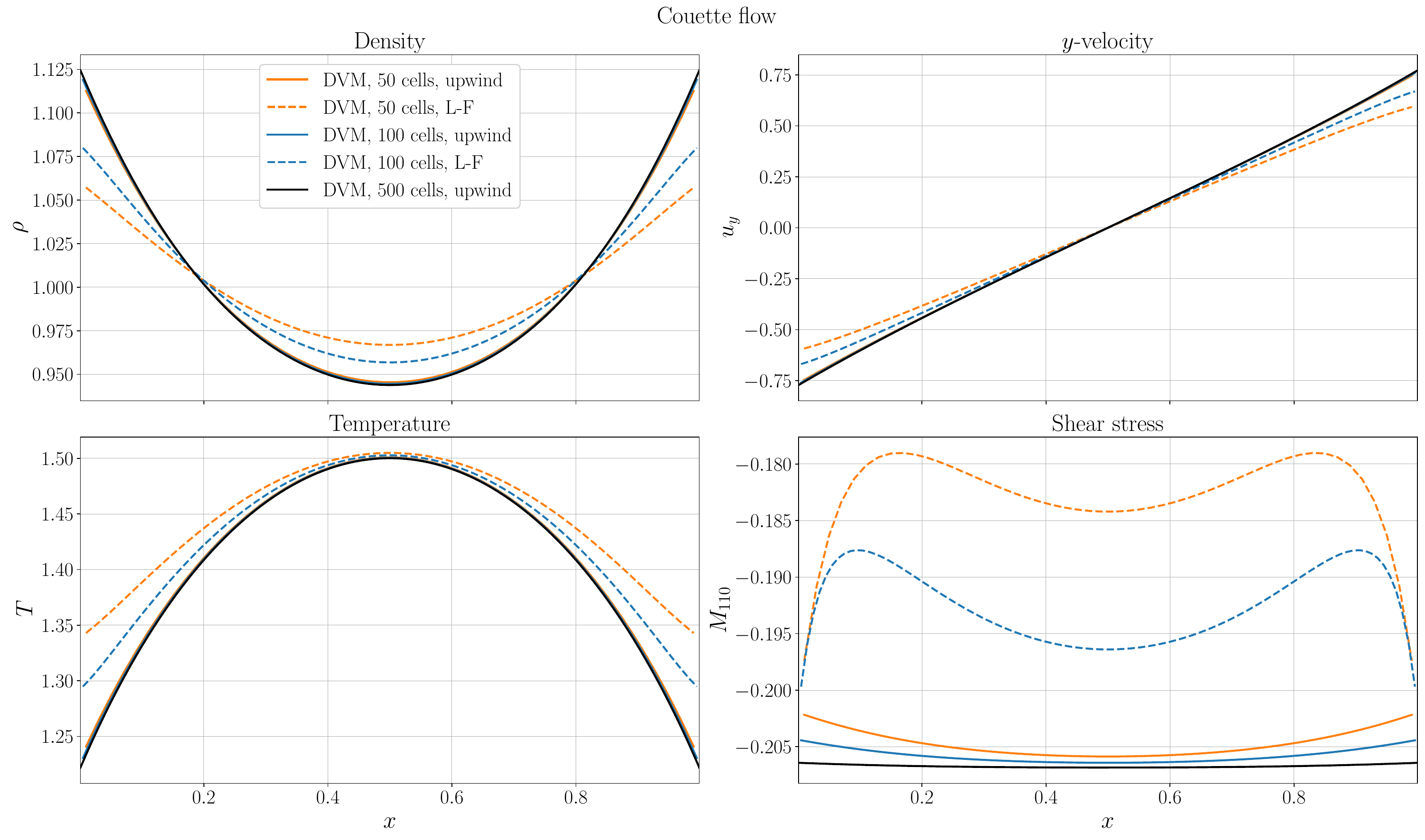}
  \caption{Couette flow profiles for density (upper left), $y$-velocity (upper right), temperature (lower left), and shear stress (lower right) computed with the DVM method.}\label{fig:couette-dvm}
  \end{figure}

Figure~\ref{fig:couette-dvm} shows the density, $y$-velocity, $T$, and shear stress profiles computed using the discrete velocity method on grids of 50, 100, and 500 cells, the latter computed only with the upwind method, and the others computed using either the upwind flux (``upwind'') or the Lax--Friedrichs flux (``L-F'').
We observe that all the upwinding results lie much closer to one another, with very slight deviations observable in the shear stress profile. The significantly higher numerical viscosity of the Lax--Friedrichs method leads to stronger variation in the solutions computed on different grids, and also leads to qualitative differences in the shear stress profile.
As such, we retain the upwinding DVM result on a 100-cell grid as the reference solution, but also retain the Lax--Friedrichs result on a 100-cell grid for a consistent comparison with results of the moment model computed using the Lax--Friedrichs method.

\begin{figure}[t!]
  \centering
  \includegraphics[width=0.88\textwidth]{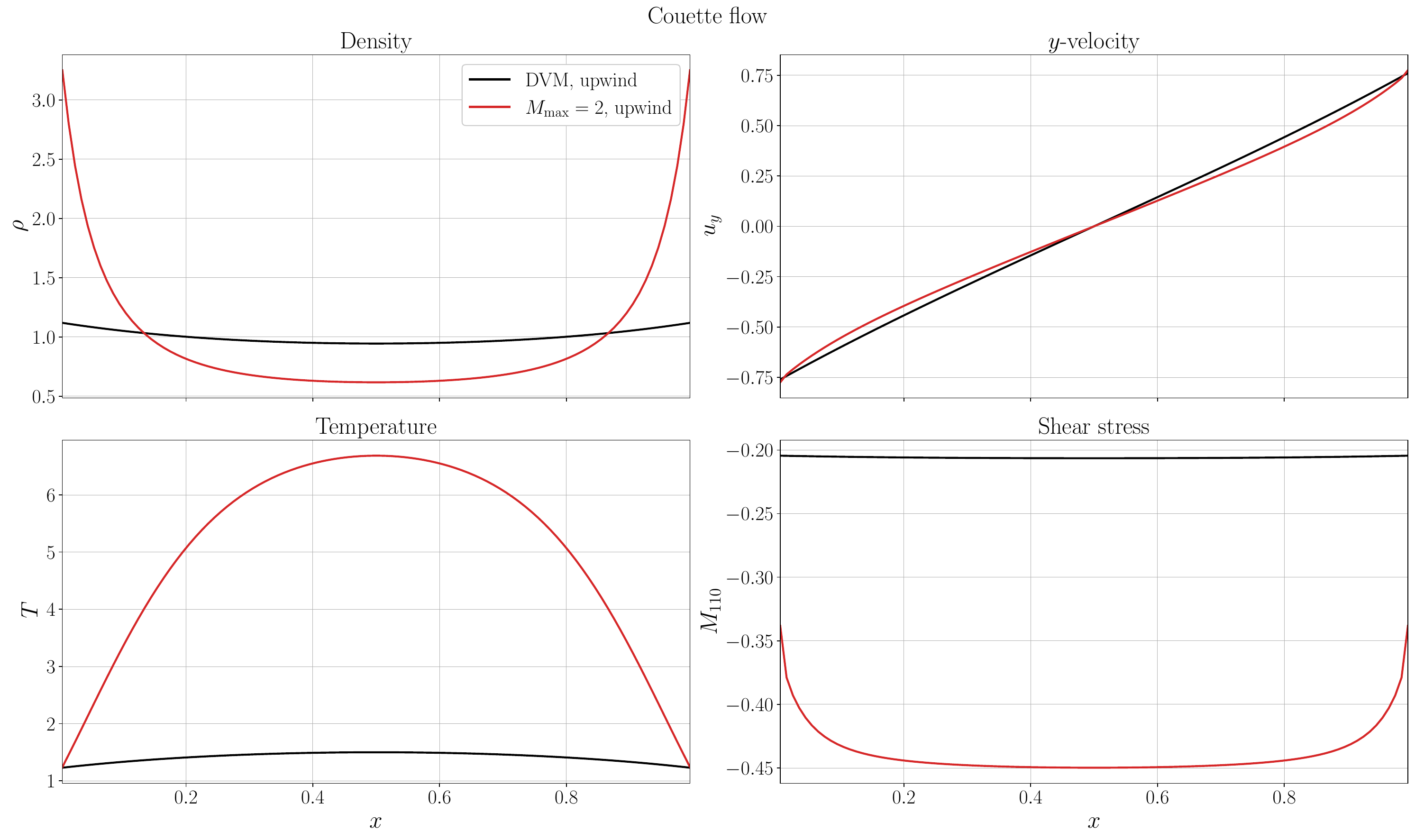}
  \caption{Couette flow profiles for density (upper left), $y$-velocity (upper right), temperature (lower left), and shear stress (lower right) computed with the DVM method and the sparse entropic quadrature method with $M_{\mathrm{max}}=2$ with $\lambda=0$.}\label{fig:couette-upwind-m2}
  \end{figure}

First, we consider the results computed with the 10-moment ($M_{\mathrm{max}}=2$) system. Figure~\ref{fig:couette-upwind-m2} shows the density, $y$-velocity, temperature, and shear stress profiles computed using the upwind scheme. We see that the lack of explicitly modelled heat transfer terms in the system has a strong detrimental effect on the solution quality, and the density, temperature and shear stress profiles exhibit extreme deviations from the reference solution. As such, we exclude the $M_{\mathrm{max}}=2$ case from further consideration.

\begin{figure}[t!]
  \centering
  \includegraphics[width=0.88\textwidth]{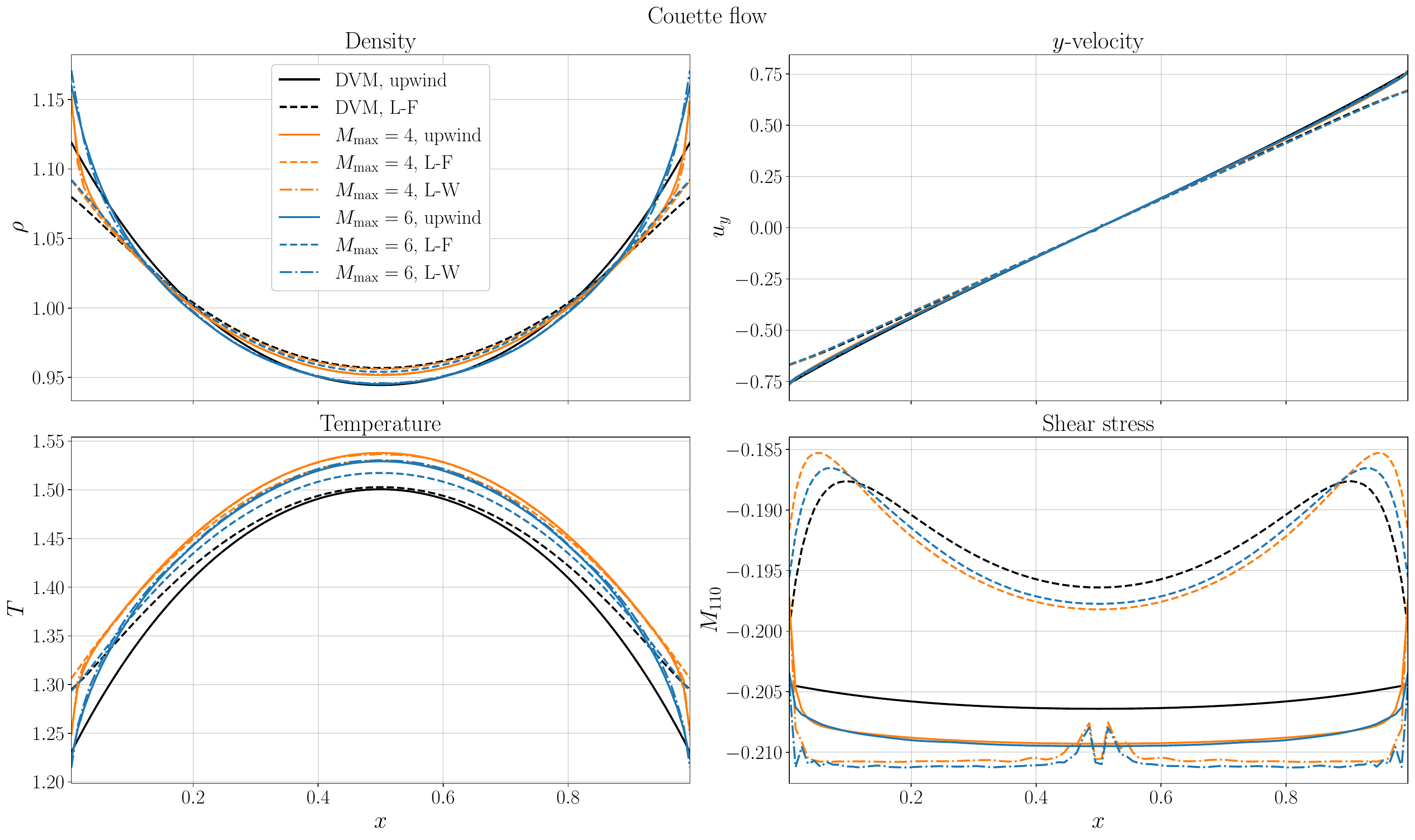}
  \caption{Couette flow profiles for density (upper left), $y$-velocity (upper right), temperature (lower left), and shear stress (lower right) computed with the DVM method and the sparse entropic quadrature method with $M_{\mathrm{max}}=4$ and $M_{\mathrm{max}}=6$ for $\lambda=0$.}\label{fig:couette-dvm-m4}
  \end{figure}

Figure~\ref{fig:couette-dvm-m4} shows the density, $y$-velocity, $T$, and shear stress profiles computed using the discrete velocity method and the sparse entropic quadrature method with $M_{\mathrm{max}}=4$ and $M_{\mathrm{max}}=6$ with no sparsity-promoting regularization, i.e. with $\lambda=0$. The solutions to moment equations were computed using the Lax-Friedrichs method (``L-F''), upwinding (``upwind''), and the Lax-Wendroff method (``L-W'').
We see that the moment method gives results close to those given by the DVM method, with the Lax-Friedrichs method also leading to stronger deviations in the shear stress profile, similar to the behaviour exhibited by the DVM solution. The higher-order Lax-Wendroff solution exhibits slight oscillations in the center of the domain.
This is due to the center of the domain exhibiting a local minimum/maximum (depending on the quantity considered), which is physical; however, the flux limiter used in the Lax-Wendroff scheme reduces the the scheme to a first-order one due to this flow feature, whereas in the neighbouring cells, the scheme is back to being second-order. This abrupt variation in the flux limiter leads to the oscillations seen in the solution, highlighting the need for more robust flux limiting approaches. We also note that both of the lower-diffusion schemes, i.e. upwinding and Lax-Wendroff, exhibit larger temperature/density and shear stress gradients at the boundaries. 
The reason is that the solution at the boundary, computed via the kinetic boundary condition, is discontinuous in velocity space. The moment method is not able to resolve this discontinuity due to the exponential structure of the solution. So reduction of numerical diffusion leads to the issue becoming more prominent, whereas with a highly diffusive scheme (Lax-Friedrichs) it is not observed. DVM, operating node-wise, does not experience the same problem.
Comparing the $M_{\mathrm{max}}=4$ and $M_{\mathrm{max}}=6$ results, we see that increasing the number of conserved moments improves the resolution of the flow in the boundary layer, as the higher-order polynomial in the exponential defining the distribution function representation~\eqref{eq:gibbs-form} is better capable of approximating abrupt jumps in the distribution.



\begin{figure}[t!]
  \centering
  \includegraphics[width=0.88\textwidth]{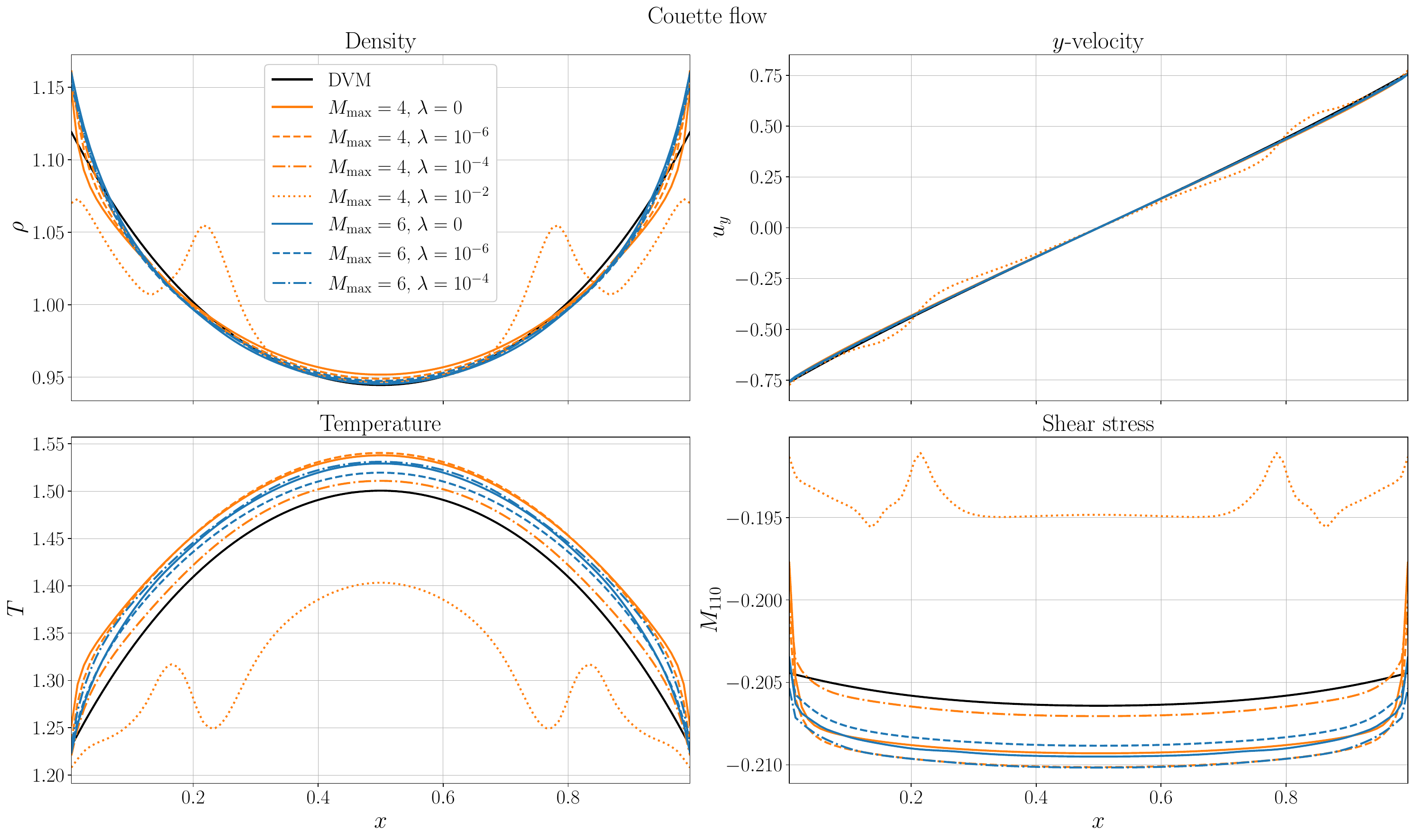}
  \caption{Couette flow profiles for density (upper left), $y$-velocity (upper right), temperature (lower left), and shear stress (lower right) computed with the DVM method and the sparse entropic quadrature method with $M_{\mathrm{max}}=4$ and $M_{\mathrm{max}}=6$ for different values of $\lambda$.}\label{fig:couette-dvm-sparse}
  \end{figure}

  Finally, we consider the impact of the sparsity-promoting regularization on the solution. Figure~\ref{fig:couette-dvm-sparse} shows the flow quantities computed for various values of the regularization parameter $\lambda$. Similar to the Sod shock tube case, for the largest value considered ($\lambda=10^{-2}$), the solution is either significantly degraded ($M_{\mathrm{max}}=4$) or fails to converge ($M_{\mathrm{max}}=6$). However, for smaller values of $\lambda$, the results are very similar to those computed with the DVM method. Analysis of the sparsity in the underlying distribution gives similar values to those seen in the Sod shock tube case: approximately 80\% of the values of the weighting function $g$ are smaller than the chosen threshold value of $10^{-10}$ for the case of $\lambda=10^{-6}$ and approximately 88-90\% of the values of $g$ are small than the threshold value when $\lambda=10^{-4}$.
  Thus, it is possible to achieve a similar level of accuracy with a significant reduction in the number of degrees of freedom in the underlying distribution, thus enabling efficient coupling with DVM solvers which solve the full Boltzmann collision term, the cost of the
  evaluation of which scales directly with the number of non-zero values of the discretized distribution.

\section{Conclusions}\label{sec:conclusions}
We have provided a theoretical analysis of the recently proposed sparse entropic quadrature method for the closure of the kinetic moment equations. We have shown that the method is well-posed and that these properties are retained under a kinetic boundary treatment coupled with appropriate convection and timestepping schemes. We have also shown that the method is able to accurately compute the solution of the kinetic moment equations for simple test cases, and that enforcing even a large degree of sparsity does not have a noteable impact on the accuracy of the solution.
Future work will focus on development of more sophisticated schemes for the treatment of kinetic boundary conditions, more robust higher-order methods for the convective term discretization, and the application of the method to more complex problems.

\backmatter

\bmhead{Acknowledgements}
The authors thank the Deutsche Forschungsgemeinschaft (DFG, German Research Foundation) for financial support through the SFB1481 ``Sparsity and Singular Structures'' (442047500) within the project B04 ``Sparsity Patterns in Kinetic Theory''.

Large Language Models (LLMs), specifically, Claude Fable 5 and Claude Opus 5 were used to assist in the writing of this paper. In the theoretical part, the assistance was restricted to literature search and suggestions on some of the aspects of the theoretical analysis presented in the paper. All of the references and statements have been independently verified by the authors of the paper. Claude Opus 5 was used to develop some parts of the code used to produce the numerical results, the code has been independently tested and verified by the authors of the paper. The authors take full responsibility for the content of the paper.


\begin{thebibliography}{63}
\ifx \bisbn   \undefined \def \bisbn  #1{ISBN #1}\fi
\ifx \binits  \undefined \def \binits#1{#1}\fi
\ifx \bauthor  \undefined \def \bauthor#1{#1}\fi
\ifx \batitle  \undefined \def \batitle#1{#1}\fi
\ifx \bjtitle  \undefined \def \bjtitle#1{#1}\fi
\ifx \bvolume  \undefined \def \bvolume#1{\textbf{#1}}\fi
\ifx \byear  \undefined \def \byear#1{#1}\fi
\ifx \bissue  \undefined \def \bissue#1{#1}\fi
\ifx \bfpage  \undefined \def \bfpage#1{#1}\fi
\ifx \blpage  \undefined \def \blpage #1{#1}\fi
\ifx \burl  \undefined \def \burl#1{\textsf{#1}}\fi
\ifx \doiurl  \undefined \def \doiurl#1{\url{https://doi.org/#1}}\fi
\ifx \betal  \undefined \def \betal{\textit{et al.}}\fi
\ifx \binstitute  \undefined \def \binstitute#1{#1}\fi
\ifx \binstitutionaled  \undefined \def \binstitutionaled#1{#1}\fi
\ifx \bctitle  \undefined \def \bctitle#1{#1}\fi
\ifx \beditor  \undefined \def \beditor#1{#1}\fi
\ifx \bpublisher  \undefined \def \bpublisher#1{#1}\fi
\ifx \bbtitle  \undefined \def \bbtitle#1{#1}\fi
\ifx \bedition  \undefined \def \bedition#1{#1}\fi
\ifx \bseriesno  \undefined \def \bseriesno#1{#1}\fi
\ifx \blocation  \undefined \def \blocation#1{#1}\fi
\ifx \bsertitle  \undefined \def \bsertitle#1{#1}\fi
\ifx \bsnm \undefined \def \bsnm#1{#1}\fi
\ifx \bsuffix \undefined \def \bsuffix#1{#1}\fi
\ifx \bparticle \undefined \def \bparticle#1{#1}\fi
\ifx \barticle \undefined \def \barticle#1{#1}\fi
\bibcommenthead
\ifx \bconfdate \undefined \def \bconfdate #1{#1}\fi
\ifx \botherref \undefined \def \botherref #1{#1}\fi
\ifx \url \undefined \def \url#1{\textsf{#1}}\fi
\ifx \bchapter \undefined \def \bchapter#1{#1}\fi
\ifx \bbook \undefined \def \bbook#1{#1}\fi
\ifx \bcomment \undefined \def \bcomment#1{#1}\fi
\ifx \oauthor \undefined \def \oauthor#1{#1}\fi
\ifx \citeauthoryear \undefined \def \citeauthoryear#1{#1}\fi
\ifx \endbibitem  \undefined \def \endbibitem {}\fi
\ifx \bconflocation  \undefined \def \bconflocation#1{#1}\fi
\ifx \arxivurl  \undefined \def \arxivurl#1{\textsf{#1}}\fi
\csname PreBibitemsHook\endcsname

\bibitem[\protect\citeauthoryear{Cercignani}{1988}]{cercignani1988}
\begin{bbook}
\bauthor{\bsnm{Cercignani}, \binits{C.}}:
\bbtitle{The {B}oltzmann Equation and Its Applications}.
\bpublisher{Springer},
\blocation{New York}
(\byear{1988})
\end{bbook}
\endbibitem

\bibitem[\protect\citeauthoryear{Shen}{2005}]{shen2005rarefied}
\begin{bbook}
\bauthor{\bsnm{Shen}, \binits{C.}}:
\bbtitle{Rarefied Gas Dynamics: Fundamentals, Simulations and Micro Flows}.
\bpublisher{Springer},
\blocation{Berlin, Heidelberg}
(\byear{2005})
\end{bbook}
\endbibitem

\bibitem[\protect\citeauthoryear{Torrilhon and Sarna}{2017}]{torrilhon2017hierarchical}
\begin{barticle}
\bauthor{\bsnm{Torrilhon}, \binits{M.}},
\bauthor{\bsnm{Sarna}, \binits{N.}}:
\batitle{Hierarchical {Boltzmann} simulations and model error estimation}.
\bjtitle{Journal of Computational Physics}
\bvolume{342},
\bfpage{66}--\blpage{84}
(\byear{2017})
\end{barticle}
\endbibitem

\bibitem[\protect\citeauthoryear{Verbiest and Koellermeier}{2026}]{verbiest2026model}
\begin{barticle}
\bauthor{\bsnm{Verbiest}, \binits{R.}},
\bauthor{\bsnm{Koellermeier}, \binits{J.}}:
\batitle{Model-error estimation and model adaptivity for hyperbolic moment equations in one dimension}.
\bjtitle{Microfluidics and Nanofluidics}
\bvolume{30}(\bissue{8}),
\bfpage{62}
(\byear{2026})
\end{barticle}
\endbibitem

\bibitem[\protect\citeauthoryear{Hamburger}{1944}]{hamburger1944hermitian}
\begin{barticle}
\bauthor{\bsnm{Hamburger}, \binits{H.L.}}:
\batitle{Hermitian transformations of deficiency-index (1, 1), {J}acobi matrices and undetermined moment problems}.
\bjtitle{Am. J. Math.}
\bvolume{66}(\bissue{4}),
\bfpage{489}--\blpage{522}
(\byear{1944})
\end{barticle}
\endbibitem

\bibitem[\protect\citeauthoryear{Shohat and Tamarkin}{1945}]{Shohat1945problem}
\begin{bbook}
\bauthor{\bsnm{Shohat}, \binits{J.}},
\bauthor{\bsnm{Tamarkin}, \binits{J.}}:
\bbtitle{The Problem of Moments}.
\bpublisher{AMS, Providence},
\blocation{Providence}
(\byear{1945})
\end{bbook}
\endbibitem

\bibitem[\protect\citeauthoryear{Schm{\"u}dgen}{2017}]{schmuedgen2017moment}
\begin{bbook}
\bauthor{\bsnm{Schm{\"u}dgen}, \binits{K.}}:
\bbtitle{The Moment Problem}.
\bsertitle{Graduate Texts in Mathematics}.
\bpublisher{Springer},
\blocation{Cham}
(\byear{2017})
\end{bbook}
\endbibitem

\bibitem[\protect\citeauthoryear{Struchtrup}{1998}]{struchtrup1998number}
\begin{barticle}
\bauthor{\bsnm{Struchtrup}, \binits{H.}}:
\batitle{On the number of moments in radiative transfer problems}.
\bjtitle{Ann. Phys. (N. Y.)}
\bvolume{266}(\bissue{1}),
\bfpage{1}--\blpage{26}
(\byear{1998})
\end{barticle}
\endbibitem

\bibitem[\protect\citeauthoryear{Modest and Mazumder}{2021}]{modest2021radiative}
\begin{bbook}
\bauthor{\bsnm{Modest}, \binits{M.F.}},
\bauthor{\bsnm{Mazumder}, \binits{S.}}:
\bbtitle{Radiative Heat Transfer}.
\bpublisher{Academic press},
\blocation{Cambridge, MA}
(\byear{2021})
\end{bbook}
\endbibitem

\bibitem[\protect\citeauthoryear{Milbrandt and Yau}{2005}]{milbrandt2005multimoment}
\begin{barticle}
\bauthor{\bsnm{Milbrandt}, \binits{J.}},
\bauthor{\bsnm{Yau}, \binits{M.}}:
\batitle{A multimoment bulk microphysics parameterization. part {II}: A proposed three-moment closure and scheme description}.
\bjtitle{J. Atmos. Sci.}
\bvolume{62}(\bissue{9}),
\bfpage{3065}--\blpage{3081}
(\byear{2005})
\end{barticle}
\endbibitem

\bibitem[\protect\citeauthoryear{Yuan et~al.}{2012}]{yuan2012extended}
\begin{barticle}
\bauthor{\bsnm{Yuan}, \binits{C.}},
\bauthor{\bsnm{Laurent}, \binits{F.}},
\bauthor{\bsnm{Fox}, \binits{R.}}:
\batitle{An extended quadrature method of moments for population balance equations}.
\bjtitle{J. Aerosol Sci.}
\bvolume{51},
\bfpage{1}--\blpage{23}
(\byear{2012})
\end{barticle}
\endbibitem

\bibitem[\protect\citeauthoryear{Koellermeier and Rominger}{2020}]{koellermeier2020analysis}
\begin{barticle}
\bauthor{\bsnm{Koellermeier}, \binits{J.}},
\bauthor{\bsnm{Rominger}, \binits{M.}}:
\batitle{Analysis and numerical simulation of hyperbolic shallow water moment equations}.
\bjtitle{Commun. Comput. Phys.}
\bvolume{28}(\bissue{3}),
\bfpage{1038}--\blpage{1084}
(\byear{2020})
\end{barticle}
\endbibitem

\bibitem[\protect\citeauthoryear{Singh and Hespanha}{2006}]{singh2006moment}
\begin{bchapter}
\bauthor{\bsnm{Singh}, \binits{A.}},
\bauthor{\bsnm{Hespanha}, \binits{J.P.}}:
\bctitle{Moment closure techniques for stochastic models in population biology}.
In: \bbtitle{2006 American Control Conference}
(\byear{2006}).
\bcomment{IEEE}
\end{bchapter}
\endbibitem

\bibitem[\protect\citeauthoryear{Gillespie}{2009}]{gillespie2009moment}
\begin{barticle}
\bauthor{\bsnm{Gillespie}, \binits{C.S.}}:
\batitle{Moment-closure approximations for mass-action models}.
\bjtitle{IET Syst. Biol.}
\bvolume{3}(\bissue{1}),
\bfpage{52}--\blpage{58}
(\byear{2009})
\end{barticle}
\endbibitem

\bibitem[\protect\citeauthoryear{Marques~Jr and M{\'e}ndez}{2013}]{marques2013kinetic}
\begin{barticle}
\bauthor{\bsnm{Marques~Jr}, \binits{W.}},
\bauthor{\bsnm{M{\'e}ndez}, \binits{A.}}:
\batitle{On the kinetic theory of vehicular traffic flow: {C}hapman--{E}nskog expansion versus {G}rad’s moment method}.
\bjtitle{Phys. A: Stat. Mech. Appl.}
\bvolume{392}(\bissue{16}),
\bfpage{3430}--\blpage{3440}
(\byear{2013})
\end{barticle}
\endbibitem

\bibitem[\protect\citeauthoryear{Herty et~al.}{2020}]{herty2020bgk}
\begin{barticle}
\bauthor{\bsnm{Herty}, \binits{M.}},
\bauthor{\bsnm{Puppo}, \binits{G.}},
\bauthor{\bsnm{Roncoroni}, \binits{S.}},
\bauthor{\bsnm{Visconti}, \binits{G.}}:
\batitle{The {BGK} approximation of kinetic models for traffic}.
\bjtitle{Kinet. Relat. Models}
\bvolume{13}(\bissue{2}),
\bfpage{279}--\blpage{307}
(\byear{2020})
\end{barticle}
\endbibitem

\bibitem[\protect\citeauthoryear{Grad}{1949}]{grad2kinetic}
\begin{botherref}
\oauthor{\bsnm{Grad}, \binits{H.}}:
On the kinetic theory of rarefied gases.
Commun. Pure Appl. Math
\textbf{2}(331)
(1949)
\end{botherref}
\endbibitem

\bibitem[\protect\citeauthoryear{Struchtrup and Torrilhon}{2003}]{struchtrup2003regularization}
\begin{barticle}
\bauthor{\bsnm{Struchtrup}, \binits{H.}},
\bauthor{\bsnm{Torrilhon}, \binits{M.}}:
\batitle{Regularization of {G}rad’s 13 moment equations: Derivation and linear analysis}.
\bjtitle{Phys. Fluids}
\bvolume{15}(\bissue{9}),
\bfpage{2668}--\blpage{2680}
(\byear{2003})
\end{barticle}
\endbibitem

\bibitem[\protect\citeauthoryear{McGraw}{1997}]{mcgraw1997description}
\begin{barticle}
\bauthor{\bsnm{McGraw}, \binits{R.}}:
\batitle{Description of aerosol dynamics by the quadrature method of moments}.
\bjtitle{Aerosol Sci. Tech.}
\bvolume{27}(\bissue{2}),
\bfpage{255}--\blpage{265}
(\byear{1997})
\end{barticle}
\endbibitem

\bibitem[\protect\citeauthoryear{Fox}{2008}]{fox2008quadrature}
\begin{barticle}
\bauthor{\bsnm{Fox}, \binits{R.O.}}:
\batitle{A quadrature-based third-order moment method for dilute gas-particle flows}.
\bjtitle{J. Comput. Phys.}
\bvolume{227}(\bissue{12}),
\bfpage{6313}--\blpage{6350}
(\byear{2008})
\end{barticle}
\endbibitem

\bibitem[\protect\citeauthoryear{Desjardins et~al.}{2008}]{desjardins2008quadrature}
\begin{barticle}
\bauthor{\bsnm{Desjardins}, \binits{O.}},
\bauthor{\bsnm{Fox}, \binits{R.O.}},
\bauthor{\bsnm{Villedieu}, \binits{P.}}:
\batitle{A quadrature-based moment method for dilute fluid-particle flows}.
\bjtitle{J. Comput. Phys.}
\bvolume{227}(\bissue{4}),
\bfpage{2514}--\blpage{2539}
(\byear{2008})
\end{barticle}
\endbibitem

\bibitem[\protect\citeauthoryear{Chalons et~al.}{2010}]{chalons2010beyond}
\begin{botherref}
\oauthor{\bsnm{Chalons}, \binits{C.}},
\oauthor{\bsnm{Kah}, \binits{D.}},
\oauthor{\bsnm{Massot}, \binits{M.}}:
Beyond pressureless gas dynamics: quadrature-based velocity moment models.
arXiv preprint arXiv:1011.2974
(2010)
\end{botherref}
\endbibitem

\bibitem[\protect\citeauthoryear{Fox et~al.}{2018}]{fox2018conditional}
\begin{barticle}
\bauthor{\bsnm{Fox}, \binits{R.O.}},
\bauthor{\bsnm{Laurent}, \binits{F.}},
\bauthor{\bsnm{Vi{\'e}}, \binits{A.}}:
\batitle{Conditional hyperbolic quadrature method of moments for kinetic equations}.
\bjtitle{J. Comput. Phys.}
\bvolume{365},
\bfpage{269}--\blpage{293}
(\byear{2018})
\end{barticle}
\endbibitem

\bibitem[\protect\citeauthoryear{Van~Cappellen et~al.}{2021}]{van2021higher}
\begin{barticle}
\bauthor{\bsnm{Van~Cappellen}, \binits{M.}},
\bauthor{\bsnm{Vetrano}, \binits{M.R.}},
\bauthor{\bsnm{Laboureur}, \binits{D.}}:
\batitle{Higher order hyperbolic quadrature method of moments for solving kinetic equations}.
\bjtitle{J. Comput. Phys.}
\bvolume{436},
\bfpage{110280}
(\byear{2021})
\end{barticle}
\endbibitem

\bibitem[\protect\citeauthoryear{Huang et~al.}{2020}]{huang2020stability}
\begin{barticle}
\bauthor{\bsnm{Huang}, \binits{Q.}},
\bauthor{\bsnm{Li}, \binits{S.}},
\bauthor{\bsnm{Yong}, \binits{W.-A.}}:
\batitle{Stability analysis of quadrature-based moment methods for kinetic equations}.
\bjtitle{SIAM J. Appl. Math.}
\bvolume{80}(\bissue{1}),
\bfpage{206}--\blpage{231}
(\byear{2020})
\end{barticle}
\endbibitem

\bibitem[\protect\citeauthoryear{Fox and Laurent}{2022}]{FoxLaurent}
\begin{barticle}
\bauthor{\bsnm{Fox}, \binits{R.O.}},
\bauthor{\bsnm{Laurent}, \binits{F.}}:
\batitle{Hyperbolic quadrature method of moments for the one-dimensional kinetic equation}.
\bjtitle{SIAM J. Appl. Math.}
\bvolume{82}(\bissue{2}),
\bfpage{750}--\blpage{771}
(\byear{2022})
\end{barticle}
\endbibitem

\bibitem[\protect\citeauthoryear{Fox et~al.}{2023}]{fox2023generalized}
\begin{barticle}
\bauthor{\bsnm{Fox}, \binits{R.O.}},
\bauthor{\bsnm{Laurent}, \binits{F.}},
\bauthor{\bsnm{Passalacqua}, \binits{A.}}:
\batitle{The generalized quadrature method of moments}.
\bjtitle{J. Aerosol Sci.}
\bvolume{167},
\bfpage{106096}
(\byear{2023})
\end{barticle}
\endbibitem

\bibitem[\protect\citeauthoryear{Yilmaz et~al.}{2026}]{yilmaz2026nonlinear}
\begin{barticle}
\bauthor{\bsnm{Yilmaz}, \binits{E.}},
\bauthor{\bsnm{Oblapenko}, \binits{G.}},
\bauthor{\bsnm{Torrilhon}, \binits{M.}}:
\batitle{On nonlinear closures for moment equations based on orthogonal polynomials}.
\bjtitle{SIAM Journal on Applied Mathematics}
\bvolume{86}(\bissue{3}),
\bfpage{839}--\blpage{869}
(\byear{2026})
\end{barticle}
\endbibitem

\bibitem[\protect\citeauthoryear{Levermore}{1997}]{levermore1997entropy}
\begin{barticle}
\bauthor{\bsnm{Levermore}, \binits{C.D.}}:
\batitle{Entropy-based moment closures for kinetic equations}.
\bjtitle{Transp. Theor. Stat.}
\bvolume{26}(\bissue{4-5}),
\bfpage{591}--\blpage{606}
(\byear{1997})
\end{barticle}
\endbibitem

\bibitem[\protect\citeauthoryear{McDonald and Torrilhon}{2013}]{mcdonald2013affordable}
\begin{barticle}
\bauthor{\bsnm{McDonald}, \binits{J.}},
\bauthor{\bsnm{Torrilhon}, \binits{M.}}:
\batitle{Affordable robust moment closures for {CFD} based on the maximum-entropy hierarchy}.
\bjtitle{J. Comput. Phys.}
\bvolume{251},
\bfpage{500}--\blpage{523}
(\byear{2013})
\end{barticle}
\endbibitem

\bibitem[\protect\citeauthoryear{Alldredge et~al.}{2019}]{alldredge2019regularized}
\begin{barticle}
\bauthor{\bsnm{Alldredge}, \binits{G.W.}},
\bauthor{\bsnm{Frank}, \binits{M.}},
\bauthor{\bsnm{Hauck}, \binits{C.D.}}:
\batitle{A regularized entropy-based moment method for kinetic equations}.
\bjtitle{SIAM Journal on Applied Mathematics}
\bvolume{79}(\bissue{5}),
\bfpage{1627}--\blpage{1653}
(\byear{2019})
\end{barticle}
\endbibitem

\bibitem[\protect\citeauthoryear{Koellermeier et~al.}{2014}]{koellermeier2014framework}
\begin{botherref}
\oauthor{\bsnm{Koellermeier}, \binits{J.}},
\oauthor{\bsnm{Schaerer}, \binits{R.P.}},
\oauthor{\bsnm{Torrilhon}, \binits{M.}}:
A framework for hyperbolic approximation of kinetic equations using quadrature-based projection methods.
Kinet. Relat. Models
\textbf{7}(3)
(2014)
\end{botherref}
\endbibitem

\bibitem[\protect\citeauthoryear{Abdelmalik and Van~Brummelen}{2016}]{abdelmalik2016moment}
\begin{barticle}
\bauthor{\bsnm{Abdelmalik}, \binits{M.}},
\bauthor{\bsnm{Van~Brummelen}, \binits{E.}}:
\batitle{Moment closure approximations of the {B}oltzmann equation based on $\varphi$-divergences}.
\bjtitle{J. Stat. Phys.}
\bvolume{164}(\bissue{1}),
\bfpage{77}--\blpage{104}
(\byear{2016})
\end{barticle}
\endbibitem

\bibitem[\protect\citeauthoryear{Pichard}{2026}]{pichardconvergence}
\begin{botherref}
\oauthor{\bsnm{Pichard}, \binits{T.}}:
Convergence analysis of moment-based approximations in kinetic theory
(2026)
\end{botherref}
\endbibitem

\bibitem[\protect\citeauthoryear{Torrilhon}{2016}]{torrilhon2016modeling}
\begin{barticle}
\bauthor{\bsnm{Torrilhon}, \binits{M.}}:
\batitle{Modeling nonequilibrium gas flow based on moment equations}.
\bjtitle{Annu. Rev. Fluid Mech.}
\bvolume{48},
\bfpage{429}--\blpage{458}
(\byear{2016})
\end{barticle}
\endbibitem

\bibitem[\protect\citeauthoryear{Pichard}{2023}]{pichard2023some}
\begin{barticle}
\bauthor{\bsnm{Pichard}, \binits{T.}}:
\batitle{Some recent advances on the method of moments in kinetic theory}.
\bjtitle{ESAIM: Proceedings and Surveys}
\bvolume{75},
\bfpage{86}--\blpage{95}
(\byear{2023})
\end{barticle}
\endbibitem

\bibitem[\protect\citeauthoryear{Schaerer et~al.}{2017}]{schaerer2017efficient}
\begin{barticle}
\bauthor{\bsnm{Schaerer}, \binits{R.P.}},
\bauthor{\bsnm{Bansal}, \binits{P.}},
\bauthor{\bsnm{Torrilhon}, \binits{M.}}:
\batitle{Efficient algorithms and implementations of entropy-based moment closures for rarefied gases}.
\bjtitle{Journal of Computational Physics}
\bvolume{340},
\bfpage{138}--\blpage{159}
(\byear{2017})
\end{barticle}
\endbibitem

\bibitem[\protect\citeauthoryear{Schaerer and Torrilhon}{2017}]{schaerer201735}
\begin{barticle}
\bauthor{\bsnm{Schaerer}, \binits{R.P.}},
\bauthor{\bsnm{Torrilhon}, \binits{M.}}:
\batitle{The 35-moment system with the maximum-entropy closure for rarefied gas flows}.
\bjtitle{European Journal of Mechanics-B/Fluids}
\bvolume{64},
\bfpage{30}--\blpage{40}
(\byear{2017})
\end{barticle}
\endbibitem

\bibitem[\protect\citeauthoryear{B{\"o}hmer and Torrilhon}{2020}]{bohmer2020entropic}
\begin{barticle}
\bauthor{\bsnm{B{\"o}hmer}, \binits{N.}},
\bauthor{\bsnm{Torrilhon}, \binits{M.}}:
\batitle{Entropic quadrature for moment approximations of the {Boltzmann-BGK} equation}.
\bjtitle{J. Comput. Phys.}
\bvolume{401},
\bfpage{108992}
(\byear{2020})
\end{barticle}
\endbibitem

\bibitem[\protect\citeauthoryear{Pichard and Laurent}{2025}]{pichard2025entropy}
\begin{botherref}
\oauthor{\bsnm{Pichard}, \binits{T.}},
\oauthor{\bsnm{Laurent}, \binits{F.}}:
On the entropy dissipation of systems of quadratures
(2025)
\end{botherref}
\endbibitem

\bibitem[\protect\citeauthoryear{Oblapenko et~al.}{2026}]{oblapenko2026sparse}
\begin{botherref}
\oauthor{\bsnm{Oblapenko}, \binits{G.}},
\oauthor{\bsnm{Theisen}, \binits{L.}},
\oauthor{\bsnm{Wilhelm}, \binits{R.-P.}},
\oauthor{\bsnm{Herty}, \binits{M.}},
\oauthor{\bsnm{Torrilhon}, \binits{M.}}:
Sparse and low-rank kinetic distribution estimation.
arXiv preprint arXiv:2606.04878
(2026)
\end{botherref}
\endbibitem

\bibitem[\protect\citeauthoryear{Junk}{1998}]{junk1998domain}
\begin{barticle}
\bauthor{\bsnm{Junk}, \binits{M.}}:
\batitle{Domain of definition of {L}evermore's five-moment system}.
\bjtitle{Journal of Statistical Physics}
\bvolume{93}(\bissue{5}),
\bfpage{1143}--\blpage{1167}
(\byear{1998})
\end{barticle}
\endbibitem

\bibitem[\protect\citeauthoryear{Hauck et~al.}{2008}]{hauck2008convex}
\begin{barticle}
\bauthor{\bsnm{Hauck}, \binits{C.D.}},
\bauthor{\bsnm{Levermore}, \binits{C.D.}},
\bauthor{\bsnm{Tits}, \binits{A.L.}}:
\batitle{Convex duality and entropy-based moment closures: Characterizing degenerate densities}.
\bjtitle{SIAM Journal on Control and Optimization}
\bvolume{47}(\bissue{4}),
\bfpage{1977}--\blpage{2015}
(\byear{2008})
\end{barticle}
\endbibitem

\bibitem[\protect\citeauthoryear{Yuan and Fox}{2011}]{yuan2011conditional}
\begin{barticle}
\bauthor{\bsnm{Yuan}, \binits{C.}},
\bauthor{\bsnm{Fox}, \binits{R.O.}}:
\batitle{Conditional quadrature method of moments for kinetic equations}.
\bjtitle{Journal of Computational Physics}
\bvolume{230}(\bissue{22}),
\bfpage{8216}--\blpage{8246}
(\byear{2011})
\end{barticle}
\endbibitem

\bibitem[\protect\citeauthoryear{Rice et~al.}{2026}]{rice2026robustly}
\begin{botherref}
\oauthor{\bsnm{Rice}, \binits{E.}},
\oauthor{\bsnm{Plante-Sabourin}, \binits{{\'E}.}},
\oauthor{\bsnm{McDonald}, \binits{J.G.}}:
Robustly hyperbolic high-order moment-closures for multidimensional gases.
Journal of Computational Physics,
115026
(2026)
\end{botherref}
\endbibitem

\bibitem[\protect\citeauthoryear{Oblapenko et~al.}{2026}]{oblapenko2026sparsekrm}
\begin{barticle}
\bauthor{\bsnm{Oblapenko}, \binits{G.}},
\bauthor{\bsnm{Torrilhon}, \binits{M.}},
\bauthor{\bsnm{Herty}, \binits{M.}}:
\batitle{Sparse reconstruction of multi-dimensional kinetic distributions}.
\bjtitle{Kinetic and Related Models}
\bvolume{20},
\bfpage{80}--\blpage{104}
(\byear{2026})
\end{barticle}
\endbibitem

\bibitem[\protect\citeauthoryear{Dimarco and Pareschi}{2014}]{dimarco2014numerical}
\begin{barticle}
\bauthor{\bsnm{Dimarco}, \binits{G.}},
\bauthor{\bsnm{Pareschi}, \binits{L.}}:
\batitle{Numerical methods for kinetic equations}.
\bjtitle{Acta Numerica}
\bvolume{23},
\bfpage{369}--\blpage{520}
(\byear{2014})
\end{barticle}
\endbibitem

\bibitem[\protect\citeauthoryear{Levermore}{1996}]{levermore1996moment}
\begin{barticle}
\bauthor{\bsnm{Levermore}, \binits{C.D.}}:
\batitle{Moment closure hierarchies for kinetic theories}.
\bjtitle{Journal of statistical Physics}
\bvolume{83}(\bissue{5}),
\bfpage{1021}--\blpage{1065}
(\byear{1996})
\end{barticle}
\endbibitem

\bibitem[\protect\citeauthoryear{Rockafellar}{1970}]{rockafellar}
\begin{bbook}
\bauthor{\bsnm{Rockafellar}, \binits{R.T.}}:
\bbtitle{Convex Analysis}.
\bpublisher{Princeton University Press},
\blocation{Princeton}
(\byear{1970})
\end{bbook}
\endbibitem

\bibitem[\protect\citeauthoryear{Mieussens}{2000}]{mieussens2000discrete}
\begin{barticle}
\bauthor{\bsnm{Mieussens}, \binits{L.}}:
\batitle{Discrete velocity model and implicit scheme for the {BGK} equation of rarefied gas dynamics}.
\bjtitle{Mathematical Models and Methods in Applied Sciences}
\bvolume{10}(\bissue{08}),
\bfpage{1121}--\blpage{1149}
(\byear{2000})
\end{barticle}
\endbibitem

\bibitem[\protect\citeauthoryear{Godlewski and Raviart}{2021}]{godlewski2013numerical}
\begin{bbook}
\bauthor{\bsnm{Godlewski}, \binits{E.}},
\bauthor{\bsnm{Raviart}, \binits{P.-A.}}:
\bbtitle{Numerical Approximation of Hyperbolic Systems of Conservation Laws},
p. \bfpage{840}.
\bpublisher{Springer},
\blocation{New York}
(\byear{2021})
\end{bbook}
\endbibitem

\bibitem[\protect\citeauthoryear{Golub and Welsch}{1969}]{golub1969calculation}
\begin{barticle}
\bauthor{\bsnm{Golub}, \binits{G.H.}},
\bauthor{\bsnm{Welsch}, \binits{J.H.}}:
\batitle{Calculation of {G}auss quadrature rules}.
\bjtitle{Mathematics of Computation}
\bvolume{23}(\bissue{106}),
\bfpage{221}--\blpage{230}
(\byear{1969})
\end{barticle}
\endbibitem

\bibitem[\protect\citeauthoryear{Gautschi}{2004}]{gautschi2004orthogonal}
\begin{bbook}
\bauthor{\bsnm{Gautschi}, \binits{W.}}:
\bbtitle{Orthogonal Polynomials: Computation and Approximation}.
\bpublisher{Oxford University Press},
\blocation{Oxford}
(\byear{2004})
\end{bbook}
\endbibitem

\bibitem[\protect\citeauthoryear{Bauer and Fike}{1960}]{bauer1960norms}
\begin{barticle}
\bauthor{\bsnm{Bauer}, \binits{F.L.}},
\bauthor{\bsnm{Fike}, \binits{C.T.}}:
\batitle{Norms and exclusion theorems}.
\bjtitle{Numerische mathematik}
\bvolume{2}(\bissue{1}),
\bfpage{137}--\blpage{141}
(\byear{1960})
\end{barticle}
\endbibitem

\bibitem[\protect\citeauthoryear{Baranger et~al.}{2019}]{baranger2019numerical}
\begin{barticle}
\bauthor{\bsnm{Baranger}, \binits{C.}},
\bauthor{\bsnm{H{\'e}rouard}, \binits{N.}},
\bauthor{\bsnm{Mathiaud}, \binits{J.}},
\bauthor{\bsnm{Mieussens}, \binits{L.}}:
\batitle{Numerical boundary conditions in finite volume and discontinuous galerkin schemes for the simulation of rarefied flows along solid boundaries}.
\bjtitle{Mathematics and Computers in Simulation}
\bvolume{159},
\bfpage{136}--\blpage{153}
(\byear{2019})
\end{barticle}
\endbibitem

\bibitem[\protect\citeauthoryear{Alldredge et~al.}{2012}]{alldredge2012high}
\begin{barticle}
\bauthor{\bsnm{Alldredge}, \binits{G.W.}},
\bauthor{\bsnm{Hauck}, \binits{C.D.}},
\bauthor{\bsnm{Tits}, \binits{A.L.}}:
\batitle{High-order entropy-based closures for linear transport in slab geometry {II}: A computational study of the optimization problem}.
\bjtitle{SIAM Journal on Scientific Computing}
\bvolume{34}(\bissue{4}),
\bfpage{361}--\blpage{391}
(\byear{2012})
\end{barticle}
\endbibitem

\bibitem[\protect\citeauthoryear{Johnson et~al.}{2023}]{johnson2023positivity}
\begin{barticle}
\bauthor{\bsnm{Johnson}, \binits{E.R.}},
\bauthor{\bsnm{Rossmanith}, \binits{J.A.}},
\bauthor{\bsnm{Vaughan}, \binits{C.}}:
\batitle{Positivity-preserving {Lax}--{Wendroff} {Discontinuous} {Galerkin} schemes for quadrature-based moment-closure approximations of kinetic models}.
\bjtitle{Journal of Scientific Computing}
\bvolume{95}(\bissue{1}),
\bfpage{19}
(\byear{2023})
\end{barticle}
\endbibitem

\bibitem[\protect\citeauthoryear{Bhatnagar et~al.}{1954}]{bhatnagar1954model}
\begin{barticle}
\bauthor{\bsnm{Bhatnagar}, \binits{P.L.}},
\bauthor{\bsnm{Gross}, \binits{E.P.}},
\bauthor{\bsnm{Krook}, \binits{M.}}:
\batitle{A model for collision processes in gases. {I}. {Small} amplitude processes in charged and neutral one-component systems}.
\bjtitle{Physical review}
\bvolume{94}(\bissue{3}),
\bfpage{511}
(\byear{1954})
\end{barticle}
\endbibitem

\bibitem[\protect\citeauthoryear{Bird}{1998}]{DSMC_Bird}
\begin{bbook}
\bauthor{\bsnm{Bird}, \binits{G.A.}}:
\bbtitle{Molecular Gas Dynamics and the Direct Simulation of Gas Flows},
\bedition{2nd} edn.
\bpublisher{Oxford University Press},
\blocation{Oxford}
(\byear{1998})
\end{bbook}
\endbibitem

\bibitem[\protect\citeauthoryear{Oblapenko}{2026a}]{oblapenko2026specqkrepro}
\begin{botherref}
\oauthor{\bsnm{Oblapenko}, \binits{G.}}:
Reproducibility repository for "Sparse entropic quadrature for moment equations".
\url{https://github.com/knstmrd/reproducibility-2026-sparse-entropic-for-moment}
(2026).
\doiurl{10.5281/zenodo.22768875}
\end{botherref}
\endbibitem

\bibitem[\protect\citeauthoryear{Oblapenko}{2026b}]{oblapenko2026specqk}
\begin{botherref}
\oauthor{\bsnm{Oblapenko}, \binits{G.}}:
SPEcQK.jl.
\url{https://github.com/ACoM-RWTH/specqk}
(2026).
\doiurl{10.5281/zenodo.20026361}
\end{botherref}
\endbibitem

\bibitem[\protect\citeauthoryear{Oblapenko}{2026c}]{oblapenko2026specqkdata}
\begin{botherref}
\oauthor{\bsnm{Oblapenko}, \binits{G.}}:
Numerical simulation data for "Sparse entropic quadrature for moment equations".
\url{https://zenodo.org/records/22768394}
(2026).
\doiurl{10.5281/zenodo.22768394}
\end{botherref}
\endbibitem

\bibitem[\protect\citeauthoryear{Boccelli et~al.}{2024}]{boccelli2024gallery}
\begin{barticle}
\bauthor{\bsnm{Boccelli}, \binits{S.}},
\bauthor{\bsnm{Giroux}, \binits{F.}},
\bauthor{\bsnm{McDonald}, \binits{J.G.}}:
\batitle{A gallery of maximum-entropy distributions: 14 and 21 moments}.
\bjtitle{Journal of Statistical Physics}
\bvolume{191}(\bissue{3}),
\bfpage{39}
(\byear{2024})
\end{barticle}
\endbibitem

\end{thebibliography}

\end{document}